\documentclass{amsart}
\usepackage{breqn}
\usepackage{schock}

\title[Intersection theory of $\oM(3,6)$]{Intersection theory of the stable pair compactification of the moduli
  space of six lines in the plane}

\author{Nolan Schock}

\begin{document}
\maketitle

\begin{abstract}
  We describe sequences of blowups of $\overline{M}_{0,5} \times
  \overline{M}_{0,5}$ and $\mathbf{P}^2 \times \mathbf{P}^2$ yielding a small
  resolution of the stable pair compactification $\overline{M}(3,6)$ of the
  moduli space $M(3,6)$ of six lines in $\mathbf{P}^2$. These blowup sequences
  can be viewed, respectively, as generalizations of Keel's and Kapranov's
  constructions of $\overline{M}_{0,n}$. We use these blowup sequences to
  describe the intersection theory of $\overline{M}(3,6)$. In particular, we
  show that the Chow ring of any small resolution of $\overline{M}(3,6)$ has a
  presentation analogous to Keel's presentation of $A^*(\overline{M}_{0,n})$,
  and the Chow ring of $\overline{M}(3,6)$ is an explicit subring of the Chow
  ring of one of these small resolutions. We also introduce higher-dimensional
  versions of the $\psi$-classes on $\overline{M}_{0,n}$, and describe their
  intersections on $\overline{M}(3,6)$. Finally, we use our results to obtain an
  independent proof of Luxton's result that $\overline{M}(3,6)$ is the log
  canonical compactification of $M(3,6)$.
\end{abstract}

\section{Introduction}

Let $\oM(r,n)$ denote the moduli space of stable hyperplane arrangements,
compactifying the space $M(r,n)$ of arrangements of $n$ hyperplanes in
$\bP^{r-1}$ in general position \cite{hackingCompactificationModuliSpace2006}.
The moduli space $\oM(r,n)$ is the natural higher-dimensional generalization of
the moduli space $\oM_{0,n}$ of stable $n$-pointed curves of genus zero; indeed
$\oM(2,n) = \oM_{0,n}$.

The present paper is devoted to the study of the intersection theory of the
first nontrivial higher-dimensional case, the moduli space $\oM(3,6)$
compactifying the space $M(3,6)$ of six lines in general position in $\bP^2$.
This space was previously studied by Luxton, whose main result is that
$\oM(3,6)$ is the log canonical compactification of $M(3,6)$
\cite{luxtonLogCanonicalCompactification2008}. More precisely, for the remainder
of the paper, by $\oM(3,6)$ we actually mean the normalization of this space.
(Luxton shows the space is actually normal, so this is not an abuse of
notation.)

\subsection{Summary of results}

\subsubsection{Blowup construction}

In Section \ref{BlowupConstructions}, we give sequences of blowups along
explicit centers of $\oM_{0,5} \times \oM_{0,5}$ and $\bP^2 \times \bP^2$
yielding the small resolution $\wM_1(3,6)$ of $\oM(3,6)$ which has fibers
$\bP^1$ over all the singular points of $\oM(3,6)$ (Theorem
\ref{BlowupConstruction}). Our constructions can be viewed, respectively, as
generalizations of Keel's and Kapranov's constructions of $\oM_{0,n}$
\cite{keelIntersectionTheoryModuli1992},
\cite{kapranovChowQuotientsGrassmannians1993}.

\subsubsection{Intersection theory}

In Sections \ref{IntersectionTheorySmall}, \ref{IntersectionTheoryM},
\ref{TautologicalClasses}, we use our blowup constructions to study the
intersection theory of $\oM(3,6)$.

In Section \ref{IntersectionTheorySmall}, we show (Theorem
\ref{ChowResolutions}) that the Chow rings of the small resolutions of
$\oM(3,6)$ admit descriptions entirely analogous to Keel's description of
$A^*(\oM_{0,n})$ \cite{keelIntersectionTheoryModuli1992}. In particular, the
Chow rings of the small resolutions are generated by the classes of the boundary
divisors, and there are two types of relations.
\begin{enumerate}
\item The \textit{linear relations} are pullbacks of the linear relations on
  $\oM_{0,4} \cong \bP^1$.
\item The \textit{multiplicative relations} are just the obvious ones $\prod D_i
  = 0$ if $\bigcap D_i = \emptyset$, where $\{D_i\}$ is any collection of
  boundary divisors.
\end{enumerate}
We describe all of these relations more explicitly as well.

In Section \ref{IntersectionTheoryM} we use the descriptions of the Chow rings
of the small resolutions to describe the Chow ring of $\oM(3,6)$ as a particular
subring of $A^*(\wM_1(3,6))$ generated by Cartier divisors on $\oM(3,6)$.
(Theorem \ref{ChowM}).

In Section \ref{TautologicalClasses}, we introduce tautological classes
$\psi_{I,i}$ on any $\oM(r,n)$, generalizing the $\psi$-classes on $\oM_{0,n}$.
We describe the intersections of these $\psi$-classes on $\oM(3,6)$ (Theorem
\ref{psiIntersectionsM36}).

\subsubsection{Birational geometry}

In Section \ref{BirationalGeometry}, we use our results from the previous
sections to obtain an independent proof of Luxton's result that $(\oM(3,6),B)$
is log canonical and $K_{\oM(3,6)} + B$ is ample and Cartier (where $B =
\oM(3,6) \setminus M(3,6)$ is the boundary)
\cite{luxtonLogCanonicalCompactification2008}.

\subsection{Moduli of stable $n$-pointed rational curves} \label{M0n}

The moduli space $\oM_{0,n}$ was first constructed by Knudsen
\cite{knudsenProjectivityModuliSpace1983}, and has since been the object of
study of numerous papers. We are interested in generalizing the following
results to the higher-dimensional case.

\subsubsection{Basic results}
$\oM_{0,n}$ is a smooth irreducible variety of dimension $n-3$.
\subsubsection{Boundary} \label{boundaryM0n}
\begin{enumerate}
\item The boundary $B = \oM_{0,n} \setminus M_{0,n}$ is a normal crossings
  divisor.
\item The irreducible boundary divisors are denoted $D_{I,J}$ for $I \coprod J =
  [n]$ a partition of $[n]$ with $\lvert I \rvert, \lvert J \rvert \geq 2$. The
  boundary divisor $D_{I,J}$ parameterizes stable curves with two irreducible
  components meeting in a node, and their degenerations. One irreducible
  component contains exactly the marked points from $I$, while the other
  contains exactly the marked points from $J$. There is an isomorphism $D_{I,J}
  \cong \oM_{0,\lvert I \rvert + 1} \times \oM_{0,\lvert J \rvert + 1}$.
\item The boundary complex of $\oM_{0,n}$ (i.e. the dual complex of the normal
  crossings divisor $B$) is the tropical Grassmannian $TG(2,n)$, the flag
  complex known as the space of phylogenetic trees
  \cite{giansiracusaDualComplexOverlineM2016}.
\item Two boundary divisors $D_{I,J}$ and $D_{I',J'}$ intersect $\iff I \subset
  I', I \subset J', J \subset I'$, or $J \subset J'$.
\end{enumerate}
\subsubsection{Recursive structure}
There are $n$ forgetful maps $f_k : \oM_{0,n} \to \oM_{0,n-1}$, forgetting the
$k$th marked point and stabilizing. These realize $\oM_{0,n}$ as the universal
family over $\oM_{0,n-1}$.
\subsubsection{Kapranov's construction} Kapranov recognized $\oM_{0,n}$ as a
Chow quotient of the Grassmannian $G(2,n)$, and used this to give a construction
of $\oM_{0,n}$ as a sequence of blowups of $\bP^{n-3}$ along smooth centers of
increasing dimension \cite{kapranovChowQuotientsGrassmannians1993}. The morphism
$\oM_{0,n} \to \bP^{n-3}$ is induced by the complete linear system associated to
a certain divisor class $\psi_i \in A^1\oM_{0,n}$, defined in \ref{tautM0n}
below.

\subsubsection{Keel's construction} Keel showed that the map $\pi : \oM_{0,n}
\to \oM_{0,n-1} \times \oM_{0,4}$, given as the product of the forgetful map
$f_n: \oM_{0,n} \to \oM_{0,n-1}$, and the map $\rho_{123n}$, forgetting all
marked points except for $1,2,3,n$, factors as a sequence of blowups along
smooth centers of codimension two \cite{keelIntersectionTheoryModuli1992}.
\subsubsection{Intersection theory} Keel used his construction of $\oM_{0,n}$ to
give a complete description of the Chow ring of $\oM_{0,n}$
\cite{keelIntersectionTheoryModuli1992}:
\[
  A^*(\oM_{0,n}) = \frac{\bZ[D_{I,J} \mid I \coprod J = [n], \lvert I \rvert,
    \lvert J \rvert \geq 2]}{\text{the following relations}}
\]

\begin{itemize}
\item (Linear relations)
  \begin{enumerate}
  \item $D_{I,J} = D_{J,I}$, and
  \item $f^*(0) = f^*(1) = f^*(\infty)$ where $f$ is any forgetful map $f:
    \oM_{0,n} \to \oM_{0,4} \cong \bP^1$. This expands to
    \[
      \sum_{\substack{i,j \in I \\ k,l \in J}} D_{I,J} = \sum_{\substack{i,k \in
          I \\ j,l \in J}} D_{I,J} = \sum_{\substack{i,l \in I \\ j,k \in J}}
      D_{I,J}.
    \]
  \end{enumerate}
\item (Multiplicative relations) $D_{I,J}D_{I',J'} = 0$ if $D_{I,J}$ and
  $D_{I',J'}$ are disjoint in $\oM_{0,n}$ (cf \ref{boundaryM0n}(4)).
\end{itemize}
This result should be interpreted as saying that $A^*(\oM_{0,n})$ is generated
by the classes of the boundary divisors, and the only relations are the obvious
ones.

\subsubsection{Tautological classes} \label{tautM0n}

For genus $g > 0$, the Chow ring of $\oM_{g,n}$ is typically far more difficult
to describe than $A^*(\oM_{0,n})$. Instead, one usually restricts their
attention to a subring $R^*(\oM_{g,n}) \subset A^*(\oM_{g,n})$, called the
tautological ring; classes in the tautological ring are called tautological
classes. Tautological classes were first studied (in the unpointed case) by
Mumford \cite{mumfordEnumerativeGeometryModuli}.

The most important tautological classes are the $\psi$-classes, $\psi_i =
c_1(\sigma_i^*(\omega_{\pi}))$, where $\pi : \oM_{g,n+1} \to \oM_{g,n}$ is the
universal family, and $\sigma_i$ are the $n$ sections of $\pi$. By
\cite{faberAlgorithmsComputingIntersection1999}, all top intersections of
tautological classes on a given $\oM_{g,n}$ can be determined by top
intersections of $\psi$-classes on \emph{all} $\oM_{g,n}$. In turn, top
intersections of $\psi$-classes are determined by Witten's conjecture
\cite{wittenTwoDimensionalGravityIntersection1991}, first proven by Kontsevich
\cite{kontsevichIntersectionTheoryModuli1992}.

The situation is far simpler on $\oM_{0,n}$. As mentioned above, the Chow ring
is fully known. The $\psi$-classes on $\oM_{0,n}$ have simple expressions as
sums of boundary divisors,
\[
  \psi_i = \sum_{i \in I, j,k \in J} D_{I,J},
\]
and their top intersections are described by the nice combinatorial formula
\cite{wittenTwoDimensionalGravityIntersection1991}
\[
  \int \psi_1^{k_1}\cdots \psi_n^{k_n} \cap [\oM_{0,n}] =
  \binom{n-3}{k_1,\ldots,k_n}.
\]

\subsubsection{Birational geometry} The birational geometry of $\oM_{0,n}$ is an
ongoing area of study, beginning in \cite{keelContractibleExtremalRays1996}. In
this paper we are concerned only with the following results.
\begin{enumerate}
\item $K_{\oM_{0,n}} = \sum_{i=2}^{\lfloor n/2 \rfloor} \left(
    \frac{i(n-i)}{n-1} - 2 \right)B_i$, where $B_i = \sum_{\lvert I \rvert = i}
  D_{I,J}$ \cite[Lemma 3.5]{keelContractibleExtremalRays1996}.
\item $(\oM_{0,n},B)$ is log canonical, and the log canonical divisor
  $K_{\oM_{0,n}} + B$ is ample \cite[Lemma
  3.6]{keelContractibleExtremalRays1996}.
\end{enumerate}

\subsection{Moduli of stable hyperplane arrangements}
A detailed introduction to the moduli spaces $\oM(r,n)$ (and their weighted
versions $\oM_{\beta}(r,n)$, constructed in
\cite{alexeevWeightedGrassmanniansStable2008}) can be found in
\cite{alexeevModuliWeightedHyperplane2015}. In this section we summarize the
current progress towards generalizing the results of Section \ref{M0n} to any
$\oM(r,n)$.

\subsubsection{Basic results}In general, $\oM(r,n)$ is reducible \cite[Section
7]{hackingCompactificationModuliSpace2006}. The closure of $M(r,n)$ in
$\oM(r,n)$ is known as the main irreducible component and denoted $\oM^m(r,n)$.
By \cite[Corollary 3.9]{hackingCompactificationModuliSpace2006}, $\oM^m(r,n)$ is
Kapranov's Chow quotient of $G(r,n)$; in particular it is reduced. The main
irreducible component has dimension $(r-1)(n-r-1)$. It was studied in more depth
in \cite{keelGeometryChowQuotients2006}.

Note that $\oM(r,r+1) = \oM^m(r,r+1)$ has dimension $0$, i.e. it is a point.
\subsubsection{Boundary}
\begin{enumerate}
\item It was shown in \cite[Theorem 3.13]{keelGeometryChowQuotients2006} that
  $\oM^m(r,n)$, together with its boundary $B = \oM^m(r,n) \setminus M(r,n)$,
  has arbitrary singularities for $r \geq 3, n \geq 9$ and $r \geq 4, n \geq 8$.
\item The stable hyperplane arrangements parameterized by the boundary of
  $\oM(r,n)$ correspond, by Kapranov's visible contour construction, to matroid
  tilings of the hypersimplex $\Delta(r,n)$. The stable hyperplane arrangements
  parameterized by the boundary of the main irreducible component $\oM^m(r,n)$
  correspond to the regular matroid tilings of $\Delta(r,n)$
  \cite{hackingCompactificationModuliSpace2006},
  \cite{alexeevModuliWeightedHyperplane2015}.
\end{enumerate}

\subsubsection{Duality} \label{duality} The linear change of coordinates $y_i =
1-x_i$ induces an involution of $\bR^n$, sending the hypersimplex
\[
  \Delta(r,n) = \{(x_1,\ldots,x_n) \in \bR^n \mid 0 \leq x_i \leq 1, \sum x_i =
  r\}
\]
to the dual hypersimplex
\[
  \Delta(n-r,n) = \{(y_1,\ldots,y_n) \in \bR^n \mid 0 \leq y_i \leq 1, \sum y_i
  = n-r\}.
\]
The hypersimplex $\Delta(r,n)$ has $n$ facets $(x_k=0)$ isomorphic to
$\Delta(r,n-1)$, and $n$ facets $(x_k=1)$ isomorphic to $\Delta(r-1,n-1)$. The
duality involution sends the facets $(x_k=0)$ to the facets $(y_k=1)$, and the
facets $(x_k=1)$ to the facets $(y_k=0)$.

The duality involution induces a natural duality isomorphism $\varphi_{r,n} :
\oM(r,n) \xrightarrow{\sim} \oM(n-r,n)$ \cite[Section
4.9]{alexeevModuliWeightedHyperplane2015}. When $n=2r$, $\varphi_{r,n} :
\oM(r,2r) \to \oM(r,2r)$ is an automorphism.

Notice in particular that $\oM(r,r+2) \cong \oM(2,r+2) = \oM_{0,r+2}$.

% \begin{proposition}
%   The duality automorphism $\varphi_{r,n} : \oM(r,2r) \to \oM(r,2r)$ is the
%   identity only when $r=2$.
% \end{proposition}
%
%\begin{proof}
%  Any $\oM(r,2r)$ has a class of divisors $D_{I,J}$ corresponding to the
%  matroid tiling
%  \[
%    \{x_{I} \leq 1\}, \{x_J \leq r-1\},
%  \]
%  of $\Delta(r,2r)$, where $I,J$ form a partition of $2r$ with $\lvert I \rvert
%  = \lvert J \rvert = r$.
%
%  The duality automorphism $\varphi_{r,2r} : \oM(r,2r) \to \oM(r,2r)$ swaps the
%  indices of these divisors, i.e. $\varphi_{r,2r}(D_{I,J}) = D_{J,I}$. When $r
%  > 2$, $D_{I,J}$ and $D_{J,I}$ are different divisors, so $\varphi_{r,2r}$
%  cannot be the identity. When $r=2$, $D_{I,J} = D_{J,I}$, and these are the
%  only boundary divisors on $\oM(2,4) = \oM_{0,4}$. Thus $\varphi_{2,4} :
%  \oM_{0,4} \to \oM_{0,4}$ is the identity.
%\end{proof}

\subsubsection{Recursive structure} \label{recursivestructure} There are three
main types of maps from a given $\oM(r,n)$ to smaller moduli spaces.
\begin{enumerate}
\item Forgetful maps $f_k : \oM^m(r,n) \to \oM^m(r,n-1)$, forgetting the $k$th
  hyperplane and stabilizing. These are defined on $\oM^m(r,n)$ by \cite[Section
  1.6]{kapranovChowQuotientsGrassmannians1993}.
\item Restriction maps $r_k : \oM(r,n) \to \oM(r-1,n-1)$, restricting to the
  $k$th hyperplane. These are defined on $\oM(r,n)$ by
  \cite{hackingCompactificationModuliSpace2006}.
\item Restriction maps $r_{I \setminus i} : \oM(r,n) \to \oM_{0,n-r+2}$ for $I
  \subset [n]$, $\lvert I \rvert = r-1$, $i \in I$. These are given by
  restriction to the stable curve $C_{I \setminus i} = \bigcap_{j \in I
    \setminus i}B_j$ on the stable hyperplane arrangement $(X,B=\sum B_i)$. Note
  that these can be obtained as compositions of the first type of restriction
  maps. In the rank 3 case, $I = \{i,j\}$ and $r_{I \setminus i} = r_j$ is the
  first type of restriction map.
\end{enumerate}

In terms of matroid tilings of the hypersimplex $\Delta(r,n)$, the forgetful map
$f_k : \oM^m(r,n) \to \oM^m(r,n-1)$ corresponds to restricting to the face
$(x_k=0)$, and the restriction map $r_k : \oM^m(r,n) \to \oM^m(r-1,n-1)$
corresponds to restricting to the face $(x_k=1)$. Since the duality isomorphism
$\varphi_{r,n} : \oM(r,n) \to \oM(n-r,n)$ swaps these faces, it follows that the
forgetful and restriction maps also commute with duality, in the sense that the
$\varphi_{r-1,n-1} \circ f_k = r_k \circ \varphi_{r,n-1}$.

\subsubsection{Generalization of Kapranov's construction}
A construction of $\oM^m(3,n)$ as a sequence of blowups of $\bP^{n-4} \times
\bP^{n-4}$ was described by Gallardo and Routis in
\cite{gallardoWonderfulCompactificationsModuli2017a}. This could be viewed as a
generalization of Kapranov's blowup construction of $\oM_{0,n}$. Gallardo and
Routis show that the centers of the blowups for their construction can in
general be reducible and non-equidimensional. (In fact, it should be expected
that the centers are quite complicated in general, since $\oM^m(r,n)$ with its
boundary has arbitrary singularities.) However, they do not explicitly describe
these centers, so their construction falls short of a complete generalization of
Kapranov's construction.

More generally, one might seek a generalization of Kapranov's construction as
sequence of blowups of $(\bP^{n-r-1})^{r-1}$ yielding $\oM^m(r,n)$ or a small
resolution, where the map $\oM^m(r,n) \to (\bP^{n-r-1})^{r-1}$ is given by a
certain divisor class $\phi_I \in A^1(\oM^m(r,n))$, generalizing the
$\psi$-classes on $\oM_{0,n}$, see \ref{tautMrn} below.

\subsubsection{Generalization of Keel's construction}
There are two candidates for a generalization of Keel's construction to any
$\oM(r,n)$.
\begin{enumerate}
\item A sequence of blowups of $\oM^m(r,n-1) \times \oM(r,r+2)$, induced by the
  morphism $\oM^m(r,n) \to \oM^m(r,n-1) \times \oM(r,r+2)$ given as the product
  of the forgetful map $f_n : \oM^m(r,n) \to \oM^m(r,n-1)$ and the map $\rho :
  \oM^m(r,n) \to \oM(r,r+2)$ forgetting all but the the first $r+1$ and the last
  hyperplane.
\item A sequence of blowups of $(\oM_{0,n-r+2})^{r-1}$, where the morphism
  $\oM^m(r,n) \to (\oM_{0,n-r+2})^{r-1}$ is the product of restriction maps
  $\prod_{i \in I} r_{I \setminus i}$. If $q_i : \oM_{0,n-r+2} \to \bP^{n-r-1}$
  is Kapranov's map associated to $\psi_i$ on $\oM_{0,n-r+2}$, then the
  composition $\oM^m(r,n) \to (\oM_{0,n-r+2})^{r-1} \to (\bP^{n-r-1})^{r-1}$ is
  the candidate map for the generalization of Kapranov's construction above.
\end{enumerate}

\subsubsection{Intersection theory}

The present paper is motivated by the desire to understand the intersection
theory of $\oM(r,n)$. Since $\oM(r,n)$ is the higher-dimensional version of
$\oM_{0,n}$, one might hope for a description of $A^*(\oM(r,n))$ analogous to
Keel's description of $\oM_{0,n}$. In principle, one could obtain such a
description by generalizing Keel's construction of $\oM_{0,n}$ as discussed
above. This would allow one to determine (a small resolution of) $\oM^m(r,n)$ as
a sequence of blowups of a variety whose Chow ring is known. If the blowups are
nice enough, then the Chow ring can be computed by Keel's formula for the Chow
ring of a blowup \cite[Theorem A.1]{keelIntersectionTheoryModuli1992}, or by
other basic results on the Chow ring of a blowup (see Section
\ref{GeneralResults}). However, since $\oM^m(r,n)$ with its boundary is
arbitrarily singular in general, it is likely difficult to determine such an
inductive presentation for any $(r,n)$. Instead, one might seek to either
replace $\oM^m(r,n)$ with the log canonical compactification $\oM^{lc}(r,n)$ of
$M(r,n)$, or to instead define and study a tautological ring for $\oM(r,n)$.

\subsubsection{Tautological classes} \label{tautMrn}

In Section \ref{TautologicalClasses}, we will define on any $\oM(r,n)$ a
collection of divisor classes $\psi_{I,i}$, for $I \subset [n]$ with $\lvert I
\rvert = r-1$, and $i \in I$, generalizing the $\psi$-classes on $\oM_{0,n}$.
The class $\psi_{I,i}$ is the first Chern class of the line bundle whose fiber
at a stable hyperplane arrangement $(S, B = \sum B_i)$ is the cotangent line to
the curve $C_{I,i} = \bigcap_{j \in I \setminus i} B_j$ at the point $B_I =
\bigcap_{j \in I} B_j$.

We will also introduce classes $\phi_I = \sum_{i \in I} \psi_{I,i}$, which can
be viewed as symmetric versions of these generalized $\psi$-classes. The class
$\phi_I$ is the first Chern class of the vector bundle whose fiber at a stable
hyperplane arrangement $(S, B = \sum B_i)$ is the cotangent space to $S$ at $B_I
= \bigcap_{j \in I} B_j$.

\subsubsection{Birational geometry} By \cite[Theorem
1.5]{keelGeometryChowQuotients2006}, the pair $(\oM^m(r,n),B = \oM^m(r,n)
\setminus M(r,n))$ is not log canonical for $r \geq 3, n \geq 9$, and $r \geq 4,
n \geq 8$. It is conjectured (\cite[Conjecture
1.6]{keelGeometryChowQuotients2006}) that $\oM^m(r,n)$ is the log canonical
compactification of $M(r,n)$ in the remaining cases $r=2$, and $(r,n) =
(3,6),(3,7),(3,8)$. This is well-known for $r=2$ (see Section \ref{M0n} above),
and was proven by Luxton for $\oM^m(3,6)$
\cite{luxtonLogCanonicalCompactification2008}, and by Corey for $\oM^m(3,7)$
\cite{coreyInitialDegenerationsGrassmannians2020}. Both Luxton and Corey use
tropical methods. The remaining case $\oM^m(3,8)$ is still open.

\subsection{Moduli of six lines in the plane}
As mentioned above, $\oM(3,4)$ is a point, and $\oM(3,5) \cong \oM_{0,5}$. The
moduli space $\oM(3,6)$, compactifying the moduli space $M(3,6)$ of six lines in
general position in $\bP^2$, is therefore the first new higher-dimensional case.
All matroid tilings of $\Delta(3,6)$ are regular \cite[Section
5.7.4]{alexeevModuliWeightedHyperplane2015}, so $\oM(3,6) = \oM^m(3,6)$ is
Kapranov's Chow quotient compactification of $M(3,6)$, and in particular is
reduced irreducible. As described in the introduction, from now on whenever we
refer to $\oM(3,6)$ we actually mean the normalization of this space. The
present article is devoted to generalizing the results of the previous
subsections to $\oM(3,6)$.

\subsubsection*{Acknowledgements}

The author is grateful for the help and support of his advisor, Valery Alexeev.

\section{Boundary and resolutions} \label{BoundaryandResolutions} Most of this
section is a summary of results due to Luxton
\cite{luxtonLogCanonicalCompactification2008}.

\subsection{Boundary}

\subsubsection{Stable hyperplane arrangements}
The stable hyperplane arrangements parameterized by $\oM(3,6)$ are all described
explicitly in \cite[Section 5.7.4]{alexeevModuliWeightedHyperplane2015}.

\subsubsection{Boundary divisors} \label{BoundaryDivisors} The boundary $B =
\oM(3,6) \setminus M(3,6)$ of $\oM(3,6)$ consists of 65 irreducible boundary
divisors. They were shown by Luxton to have the following forms \cite[Section
4.2.4]{luxtonLogCanonicalCompactification2008}.

\begin{figure}[h]
  \centering \includegraphics[scale=1]{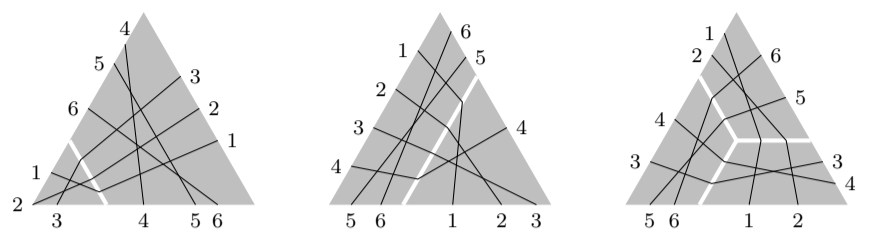}
  \caption{{\small From left to right: $D_{456,123}$, $D_{56,1234}$,
      $D_{12,34,56}$. Pictures taken from \cite[Figure
      5.12]{alexeevModuliWeightedHyperplane2015}.}}
  \label{fig:bdrydivspicture}
\end{figure}

\begin{enumerate}
\item 20 divisors $D_{ijk,lmn} \cong \oM_{0,6}$, corresponding to matroid
  tilings $\{x_{ijk} \leq 1\}, \{x_{lmn} \leq 2\}$ of $\Delta(3,6)$.
\item 15 divisors $D_{ij,klmn} \cong \oM(3,5) \times_{\oM_{0,4}} \oM_{0,5}$,
  corresponding to matroid tilings $\{x_{ij} \leq 1\}, \{x_{klmn} \leq 2\}$.
\item 30 divisors $D_{ij,kl,mn} \cong (\oM_{0,4})^3$, corresponding to matroid
  tilings $\{x_{ij} \leq 1, x_{ijkl} \leq 2\}, \{x_{kl} \leq 1, x_{klmn} \leq
  2\}, \{x_{mn} \leq 1, x_{ijmn} \leq 2\}$. The indices of these divisors are
  defined up to cyclic permutation, $D_{ij,kl,mn} = D_{kl,mn,ij} =
  D_{mn,ij,kl}$.
\end{enumerate}

\subsubsection{Non-normal crossing points}
The boundary is not a normal crossings divisor at 15 points
\[
  P_{ij,kl,mn} = D_{ij,kl,mn} \cap D_{ij,mn,kl} = D_{ij,klmn} \cap D_{kl,ijmn}
  \cap D_{mn,ijkl}.
\]

\subsubsection{Boundary complex} \label{BoundaryComplex} By
\cite{luxtonLogCanonicalCompactification2008}, the boundary complex of
$\oM(3,6)$ (i.e. the dual complex of the boundary) is a flag complex $\Delta$
constructed \cite[Section 5]{speyerTropicalGrassmannian2004} in connection with
the tropical Grassmannian $TG(3,6)$.

The vertices of $\Delta$ are labeled $e_{ijk}$ (corresponding to $D_{ijk,lmn}$),
$f_{ij}$ (corresponding to $D_{ij,klmn}$), and $g_{ij,kl,mn}$ (corresponding to
$D_{ij,kl,mn}$). Then $\Delta$ is the flag complex on the graph whose edges are
as follows.
\begin{enumerate}
\item 90 edges like $\{e_{123},e_{145}\}$ and 10 edges like
  $\{e_{123},e_{456}\}$.
\item 45 edges like $\{f_{12},f_{34}\}$.
\item 15 edges like $\{g_{12,34,56},g_{12,56,34}\}$.
\item 60 edges like $\{e_{123},f_{45}\}$ and 60 edges like $\{e_{123},f_{12}\}$.
\item 180 edges like $\{e_{123},g_{12,34,56}\}$.
\item 90 edges like $\{f_{12},g_{12,34,56}\}$.
\end{enumerate}
There are fifteen four-dimensional simplices like
\[
  \{f_{12},f_{34},f_{56},g_{12,34,56},g_{12,56,34}\}
\]
corresponding to the 15 points where the boundary is not a normal crossings
divisor.

The vertices of $\Delta$ naturally lie in $\bR^{14}$, but the abstract
simplicial complex $\Delta$ is not embedded as a simplicial complex in
$\bR^{14}$. This is because the vertices of the four-dimensional simplices only
span a three-dimensional space, the triangular bipyramid. See \cite[Section
5]{speyerTropicalGrassmannian2004} or
\cite{luxtonLogCanonicalCompactification2008} for more details.

\begin{proposition}
  The simplicial complex $\Delta$ is simply connected, and its (reduced)
  homology is concentrated in degree 3, $H_3(\Delta) = \bZ^{126}$.
\end{proposition}

\begin{proof}
  This can be calculated using a computer (cf. also \cite[Theorem
  5.4]{speyerTropicalGrassmannian2004}).
\end{proof}

We note that the boundary complex can be computed by hand via an examination of
all the stable hyperplane arrangements appearing in $\oM(3,6)$ \cite[Section
5.7.4]{alexeevModuliWeightedHyperplane2015}. In particular one can verify
independent of Luxton's methods that $\Delta$ is the boundary complex of
$\oM(3,6)$.

\subsubsection{Singularities} \label{singularities}

\begin{proposition} \label{singularitiesDescription} The moduli space $\oM(3,6)$
  has 15 singular points at the 15 points $P_{ij,kl,mn}$ where the boundary is
  not a normal crossings divisor. Each singular point looks like the origin in
  the cone over the Segre embedding of $\bP^1 \times \bP^2$.
\end{proposition}

\begin{proof}
  This was shown by Luxton using tropical methods
  \cite{luxtonLogCanonicalCompactification2008}. We will give an independent
  proof in Section \ref{IndependentLuxton} below.
\end{proof}

Locally, we can write $P_{ij,kl,mn}$ as the origin in
\[
  U = \Spec
  k[z_1,z_2,z_3,z_4,z_5,z_6]/(z_1z_5-z_2z_4,z_1z_6-z_3z_4,z_2z_6-z_3z_5),
\]
and in this neighborhood we can write
\begin{align*}
  D_{ij,klmn} = V(z_1,z_4), \;\; D_{kl,ijmn} = V(z_2,z_5), \;\; D_{mn,ijkl} = V(z_3,z_6),\\
  D_{ij,kl,mn} = V(z_1,z_2,z_3),\;\;  D_{kl,ij,mn} = V(z_4,z_5,z_6).
\end{align*}

\subsubsection{Boundary Cartier divisors} \label{Cartierdivs}

The divisors of the form $D_{ijk,lmn}$ are disjoint from the singularities of
$\oM(3,6)$, hence are Cartier.

The divisors of the forms $D_{ij,klmn}$ and $D_{ij,kl,mn}$ are not Cartier; by
the previous section they cannot be written as the vanishing locus of a single
equation in a neighborhood $U$ of $P_{ij,kl,mn}$. On the other hand, in $U$ we
compute
\[
  D_{ij,klmn} + D_{ij,kl,mn} = V(z_4),
\]
so the divisor $D_{ij,klmn} + D_{kl,ij,mn}$ is Cartier near $P_{ij,kl,mn}$.
However, it is still not Cartier on $\oM(3,6)$, because near the singular points
$P_{ij,km,ln}$ and $P_{ij,kn,lm}$, it still just looks like $D_{ij,klmn}$. It
follows that the divisors of the form
\[
  D_{ij,klmn} + D_{kl,ij,mn} + D_{km,ij,ln} + D_{kn,ij,lm}
\]
are Cartier on $\oM(3,6)$.

Also, we can write
\[
  D_{ij,kl,mn} - D_{kl,ij,mn} = (D_{ij,klmn} + D_{ij,kl,mn}) - (D_{ij,klmn} +
  D_{kl,ij,mn}),
\]
and this right-hand-side is Cartier near $P_{ij,kl,mn}$ (by the above remarks)
and does not pass through the other singular points of $\oM(3,6)$, hence
$D_{ij,kl,mn} - D_{ij,mn,kl}$ is Cartier.

Summarizing, we have the following result.

\begin{proposition} \label{BoundaryDivs}\
  \begin{enumerate}
  \item All divisors of the forms
    \begin{itemize}
    \item $D_{ijk,lmn}$,
    \item $D_{ij,klmn} + D_{kl,ij,mn} + D_{km,ij,n} + D_{kn,ij,lm}$,
    \item $D_{ij,kl,mn} - D_{ij,mn,kl}$.
    \end{itemize}
    are Cartier on $\oM(3,6)$.
  \item $D_{ij,kl,mn}$ is not Cartier only at the point $P_{ij,kl,mn}$.
  \item $D_{ij,klmn}$ is not Cartier only at the points
    $P_{ij,kl,mn},P_{ij,km,ln},P_{ij,kn,lm}$.
  \item The divisors of the forms $D_{ij,klmn} + D_{ij,kl,mn}$ are Cartier near
    $P_{ij,kl,mn}$, but not Cartier near the other singular points which they
    pass through.
  \end{enumerate}
\end{proposition}

We will later see (Theorem \ref{ChowM}) that all Cartier divisors on $\oM(3,6)$
are linear combinations of the Cartier divisors listed in part (1) of the above
proposition.

\subsection{Resolutions} \label{resolutions}

\begin{lemma} \label{LocalSmallResolutions} Let $U$ be a neighborhood of
  $P_{ij,kl,mn}$ as in Section \ref{singularities}.
  \begin{enumerate}
  \item The blowup of $U$ along $D_{ij,klmn} + D_{kl,ijmn} + D_{mn,ijkl}$ is the
    same as the blowup of $U$ along any one of these divisors. It gives the
    small resolution of $U$ with fiber $\bP^1$ over $P_{ij,kl,mn}$.
  \item The blowup of $U$ along $D_{ij,kl,mn} + D_{ij,mn,kl}$ is the same as the
    blowup of $U$ along any one of these divisors. It gives the small resolution
    of $U$ with fiber $\bP^2$ over $P_{ij,kl,mn}$.
  \end{enumerate}
\end{lemma}

\begin{proof}
  This is a local computation which we omit.
\end{proof}

For any singular point $P_{ij,kl,mn} \in \oM(3,6)$, let $I_{ij,kl,mn}^{\alpha}$,
$\alpha=1,2$ be an ideal sheaf with support $P_{ij,kl,mn}$ such that
$Bl_{I_{ij,kl,mn}^{\alpha}}U$ is the small resolution of $P_{ij,kl,mn}$ with
fiber $\bP^{\alpha}$. (This is possible by the above lemma.)

\begin{proposition} \label{SmallResolutions} Let $S_1,S_2$ be any partition of
  the set of singular points of $\oM(3,6)$. with its reduced scheme structure.
  Let $\wM_{S_1,S_2}(3,6)$ be the blowup of $\oM(3,6)$ along the ideal sheaf
  \[
    I_{S_1,S_2} = \left( \prod_{P_{ij,kl,mn} \in S_1} I_{ij,kl,mn}^1
    \right)\left( \prod_{P_{ij,kl,mn} \in S_2} I_{ij,kl,mn}^2\right).
  \]
  \begin{enumerate}
  \item $\wM_{S_1,S_2}(3,6)$ is a small resolution of $\oM(3,6)$ with fibers
    \begin{align*}
      L_{ij,kl,mn} &\cong \bP^1 \text{ over } P_{ij,kl,mn} \in S_1,\\
      \Pi_{ij,kl,mn} &\cong \bP^2 \text{ over } P_{ij,kl,mn} \in S_2.
    \end{align*}
  \item The boundary divisors in $\wM_{S_1,S_2}(3,6)$ look like
    \begin{align*}
      D_{ijk,lmn} &\cong \oM_{0,6}, \\
      D_{ij,klmn} &\cong \text{ a small resolution of } \oM(3,5) \times_{\oM_{0,4}} \oM_{0,5} \\
      D_{ij,kl,mn} &\cong
                     \begin{cases}
                       (\bP^1)^3, & P_{ij,kl,mn} \in S_1,\\
                       Bl_1(\bP^1)^3, & P_{ij,kl,mn} \in S_2.
                     \end{cases}
    \end{align*}
  \item The exceptional subvarieties are complete intersections:
    \begin{align*}
      L_{ij,kl,mn} &= D_{ij,klmn} \cap D_{kl,ijmn} \cap D_{mn,ijkl},\\
      \Pi_{ij,kl,mn} &= D_{ij,kl,mn} \cap D_{ij,mn,kl}.
    \end{align*}
  \item The boundary of $\wM_{S_1,S_2}(3,6)$ is a normal crossings divisor.
  \item
    \begin{enumerate}
    \item The boundary complex $\widetilde{\Delta}_{S_1,S_2}$ of
      $\wM_{S_1,S_2}(3,6)$ is obtained from the boundary complex $\Delta$ of
      $\oM(3,6)$ by removing the edges $\{g_{ij,kl,mn},g_{ij,mn,kl}\}$ and all
      corresponding higher-dimensional simplices, for $P_{ij,kl,mn} \in S_1$,
      and removing the triangles $\{f_{ij},f_{kl},f_{mn}\}$ and all
      corresponding higher-dimensional simplices, for $P_{ij,kl,mn} \in S_2$.
    \item $\widetilde{\Delta}_{S_1,S_2}$ is simply connected.
    \item The (reduced) homology of $\widetilde{\Delta}_{S_1,S_2}$ is
      concentrated in degree $3$; $H_3(\widetilde{\Delta}_{S_1,S_2}) =
      \bZ^{126}$.
    \end{enumerate}
  \end{enumerate}
\end{proposition}

\begin{proof}
  Part (1) is immediate from the remarks preceding the proposition.

  The remaining statements in the proposition are immediately verified from the
  local computations as in the previous lemma. The calculation of the homology
  and fundamental group of the boundary complex is performed on a computer.
\end{proof}

The two main small resolutions are $\wM_1(3,6) =
\wM_{\{P_{ij,kl,mn}\},\emptyset}(3,6)$, with fibers $\bP^1$ over all singular
points, and $\wM_2(3,6) = \wM_{\emptyset,\{P_{ij,kl,mn}\}}(3,6)$, with fibers
$\bP^2$ over all singular points. These resolutions were studied by Luxton
\cite{luxtonLogCanonicalCompactification2008}.

The boundary complex $\widetilde{\Delta}_1$ of $\wM_1(3,6)$ is the flag complex
on the graph obtained from the 1-skeleton of $\Delta$ by removing the 15 edges
$\{g_{ij,kl,mn},g_{ij,mn,kl}\}$. This is the only small resolution whose
boundary complex is a flag complex, which perhaps explains why (as we will see
below), $\wM_1(3,6)$ is easier to understand than the other small resolutions.

The boundary complex $\widetilde{\Delta}_2$ of $\wM_2(3,6)$ is the simplicial
complex obtained from $\Delta$ by removing the triangles
$\{f_{ij},f_{kl},f_{mn}\}$ and corresponding higher-dimensional simplices.
Observe that $\widetilde{\Delta}_2$ is the tropical Grassmannian $TG(3,6)$
\cite[Section 5]{speyerTropicalGrassmannian2004}.

Observe that if we blowup all the exceptional lines and planes of any small
resolution, we obtain the resolution of $\oM(3,6)$ with fibers $\bP^1 \times
\bP^2$ over all singular points. In practice we will use Lemma
\ref{LocalSmallResolutions} to obtain $\wM_1(3,6)$, and use this big resolution
just defined (or a partial version) to concretely move from one small resolution
to another.

\subsubsection{Restriction to exceptional locus}

We will use the following proposition repeatedly.
\begin{proposition} \label{RestrictionMaps}\
  \begin{enumerate}
  \item
    \begin{enumerate}
    \item $A^*(L_{ij,kl,mn}) = \bZ[p]/(p^2)$.
    \item The restriction map $A^*(\wM_{S_1,S_2}(3,6)) \to A^*(L_{ij,kl,mn})$ is
      given by
      \begin{align*}
        D_{abc,def}\vert_{L_{ij,kl,mn}} &= 0,\\
        D_{ab,cdef}\vert_{L_{ij,kl,mn}} &=
                                          \begin{cases}
                                            -p, & \text{ if } ab=ij,kl \text{ or } mn,\\
                                            0, & \text{ otherwise,}
                                          \end{cases}\\
        D_{ab,cd,ef}\vert_{L_{ij,kl,mn}} &=
                                           \begin{cases}
                                             p, & \text{ if } \{ab,cd,ef\} = \{ij,kl,mn\},\\
                                             0, & \text{ otherwise.}
                                           \end{cases}
      \end{align*}
    \end{enumerate}
  \item
    \begin{enumerate}
    \item $A^*(\Pi_{ij,kl,mn}) = \bZ[h]/(h^3)$.
    \item The restriction map $A^*(\wM_{S_1,S_2}(3,6)) \to A^*(\Pi_{ij,kl,mn})$
      is given by
      \begin{align*}
        D_{abc,def}\vert_{\Pi_{ij,kl,mn}} &= 0,\\
        D_{ab,cdef}\vert_{\Pi_{ij,kl,mn}} &=
                                            \begin{cases}
                                              h, & \text{ if } ab=ij,kl \text{ or } mn,\\
                                              0, & \text{ otherwise,}
                                            \end{cases}\\
        D_{ab,cd,ef}\vert_{\Pi_{ij,kl,mn}} &=
                                             \begin{cases}
                                               -h, & \text{ if } \{ab,cd,ef\} = \{ij,kl,mn\},\\
                                               0, & \text{ otherwise.}
                                             \end{cases}
      \end{align*}
    \end{enumerate}
  \end{enumerate}
\end{proposition}

\begin{proof}
  The formulas for the Chow rings of the exceptional lines and planes are
  immediate. The formulas for the restrictions of the boundary divisors also
  follow from Proposition \ref{SmallResolutions}. The result follows once we
  show that $A^*(\wM_{S_1,S_2}(3,6))$ is generated by the classes of the
  boundary divisors, which will be done in Section \ref{IntersectionTheorySmall}
  below.
\end{proof}

\section{Blowup constructions} \label{BlowupConstructions}

\subsection{Recursive structure}

\subsubsection{Forgetful and restriction morphisms} \label{forgetandrestrict}

In the remainder of this paper we will make frequent use of the forgetful maps
$f_k : \oM(3,6) \to \oM(3,5)$ and $r_k : \oM(3,6) \to \oM_{0,5}$. Recall that
these maps commute with duality, i.e. $\varphi_{2,5} \circ r_k = f_k \circ
\varphi_{3,6}$. The duality automorphism $\varphi_{3,6} : \oM(3,6) \to \oM(3,6)$
is nontrivial here: we have
\[
  \varphi_{3,6}(D_{ijk,lmn}) = D_{lmn,ijk},\;\; \varphi_{3,6}(D_{ij,kl,mn}) =
  D_{kl,ij,mn}.
\]

Any divisor on $\oM(3,5) \cong \oM_{0,5}$ can be written as $D_{ij,klm}$; for
notational convenience throughout the rest of this paper we will denote such a
divisor by simply $D_{ij}$.

\begin{proposition} \label{PullbackFormulas}\ We have
  \begin{align*}
    f_k^*(D_{ij}) &= D_{ijk,lmn} + D_{ij,klmn} + D_{ij,kl,mn} + D_{ij,km,ln} + D_{ij,kn,lm}, \\
    r_k^*(D_{ij}) &= D_{lmn,ijk} + D_{ij,klmn} + D_{kl,ij,mn} + D_{km,ij,ln} + D_{kn,ij,lm}.
  \end{align*}
\end{proposition}

\begin{proof}
  By duality, the claim for $f_k$ follows from the claim for $r_k$.

  Since $r_k(M(3,6)) = M_{0,5}$, we can restrict our attention to the boundary.
  Using the modular interpretation of $r_k : \oM(3,6) \to \oM_{0,5}$, one can
  give a pointwise description of $r_k$ on each boundary divisor. We find that
  $r_k(D) \subset D_{ij} \iff D$ is one of the boundary divisors in the desired
  expression for $r_k^*(D_{ij})$, so set-theoretically the desired equality
  holds. It remains to show that $r_k^*(D_{ij})$ is reduced.

  For concreteness let us work with $r_6^*(D_{12})$, which we write as
  \begin{align*}
    r_6^*(D_{12}) &= a_1D_{345,126} + a_2D_{12,3456} + a_3(D_{36,12,45} + D_{46,12,35} + D_{56,12,34}).
  \end{align*}
  (By symmetry, the coefficient on any divisor $D_{ij,kl,mn}$ is the same.)

  Our goal is to show that all $a_i = 1$.

  Observe that if $\pi: \wM_{S_1,S_2}(3,6) \to \oM(3,6)$ is any small
  resolution, and $D$ denotes the strict transform in $\wM_{S_1,S_2}(3,6)$ of a
  divisor $D$ on $\oM(3,6)$, then
  \[
    \pi^*r_6^*(D_{12}) = a_1D_{345,126} + a_2D_{12,3456} + a_3(D_{36,12,45} +
    D_{46,12,35} + D_{56,12,34}),
  \]
  so instead of working on $\oM(3,6)$ we can work on a small resolution, say
  $\wM_1(3,6)$.

  Now if $L = L_{12,34,56} = D_{12,3456} \cap D_{34,1256} \cap D_{56,1234}$ in
  $\wM_1(3,6)$, then by Proposition \ref{RestrictionMaps},
  \[
    L \cdot D_{12,3456} = -1, \;\; L \cdot D_{56,12,34} = 1,
  \]
  and $L \cdot D = 0$ for all other $D$ appearing in $r_6^*(D_{12})$, so
  \[
    L \cdot r_6^*(D_{12}) = -a_2 + a_3.
  \]
  By the projection formula, $L \cdot r_6^*(D_{12}) = 0$, so $-a_2 + a_3 = 0$,
  hence $a_2=a_3$.

  Next, if $C$ is the complete intersection $D_{123,456} \cap D_{345,126} \cap
  D_{246,135} \cong \bP^1$, then
  \begin{align*}
    C \cdot D_{345,126} &= -1, \;\; C \cdot D_{46,12,35} = 1,
  \end{align*}
  and $C \cdot D = 0$ for the other divisors in $r_6^*(D_{12})$, so
  \[
    C \cdot r_6^*(D_{12}) = -a_1 + a_3.
  \]
  Furthermore, $r_6(C)$ is a point, so by the projection formula $C \cdot
  r_6^*(D_{12}) = 0$, so $a_1 = a_3$, hence $a_1=a_2=a_3$.

  Finally, let $C' = D_{12,3456} \cap D_{56,1234} \cap D_{56,12,34}$. Then
  \begin{align*}
    C' \cdot D_{12,3456} &= 0, \text{ because } (D_{12,3456})^2\vert_{D_{56,12,34}} = 0 \text{ by Proposition \ref{RestrictionMaps}},\\
    C' \cdot D_{56,12,34} &= -1, \text{ and } C' \cdot D = 0 \text{ for all other } D \text{ in } r_6^*(D_{12}).
  \end{align*}
  We conclude that
  \[
    C' \cdot r_6^*(D_{12}) = -a_3.
  \]
  On the other hand,
  \[
    (r_6)_*(C') = D_{12},
  \]
  so by the projection formula,
  \[
    C' \cdot r_6^*(D_{12}) = D_{12}^2 = -1,
  \]
  hence $a_3=1$. Since $a_1=a_2=a_3$, this gives the result.
\end{proof}

\subsubsection{Birational morphisms} \label{birmaps}

Recall that Kapranov defines birational morphisms $q_i : \oM_{0,n} \to
\bP^{n-3}$ associated to the $i$th $\psi$-class $\psi_i$ on $\oM_{0,n}$
\cite[Chapter 4]{kapranovChowQuotientsGrassmannians1993}. On the interior these
morphisms are defined by setting the $i$th marked point to $\infty \in \bP^1$.
Kapranov factors $q_i$ as a sequence of blowups along $n-1$ points in general
position in $\bP^{n-3}$, and all the subspaces spanned by these points. For
$\oM_{0,5}$ this amounts to blowing up four points in $\bP^2$.

Let $q_{ij} : \oM(3,6) \xrightarrow{r_j \times r_i} \oM_{0,5} \times \oM_{0,5}
\xrightarrow{q_i \times q_j} \bP^2 \times \bP^2$. The morphism $q_{ij}$ is an
isomorphism on the interior. Indeed, any point in $M(3,6)$ can be represented by
a $3 \times 6$ matrix of full rank with nonzero columns; the entries of the
columns give the coefficients of the lines in $\bP^2$. The morphism $q_{ij}$ is
determined by setting the $i$th and $j$th columns of the matrix to $[0,1,0]^T$
and $[0,0,1]^T$. The remaining columns can be scaled, so that for instance
$q_{56}$ is defined on the interior by the matrix
\[
  \begin{bmatrix}
    1 & 1 & 1 & 1 & 0 & 0 \\
    x_1 & x_2 & x_3 & x_4 & 1 & 0 \\
    y_1 & y_2 & y_3 & y_4 & 0 & 1
  \end{bmatrix}
\]
giving a point in $\bP^2 \times \bP^2$ with coordinates $(x_1,x_2,x_3,x_4)
\times (y_1,y_2,y_3,y_4)$, $\sum x_i = \sum y_i = 0$. With these coordinates
$q_{56}$ is just the identity on $M(3,6) \subset \bP^2 \times \bP^2$, and
equations for any other $q_{ij}$ can be found by changing coordinates so that
the $i$th and $j$th columns have the appropriate form.

The behavior of $q_{ij}$ on the boundary strata can be determined as follows.
Kapranov's morphism $q_j : \oM_{0,n} \to \bP^{n-3}$ can be interpreted as a
weight-reduction morphism $\oM_{0,n} \to \oM_{0,\beta_j}$, where $\beta_j =
(b_1,\ldots,b_n)$ with $b_j=1$, $\sum_{i \neq j} b_i > 1$, $\sum_{i \neq j,k}
b_i < 1$ \cite[Section 6]{hassettModuliSpacesWeighted2003}. This implies that
$q_{ij} : \oM(3,6) \to \bP^{n-3}$ can be interpreted as a weight-reduction
morphism $\oM(3,6) \to \oM_{\beta_{ij}}(3,6)$ for $\beta_{ij} =
(b_1,\ldots,b_6)$ with $b_i=b_j=1$, $\sum_{k \neq i,j} b_k > 1$, $\sum_{k \neq
  i,j,l} b_k < 1$, see \cite{gallardoWonderfulCompactificationsModuli2017a}. A
boundary stratum of $\oM(3,6)$ corresponds to a matroid tiling of $\Delta(3,6)$,
while a boundary stratum of $\oM(3,6)$ corresponds to a matroid tiling of the
cut hypersimplex $\Delta_{\beta_{ij}}(3,6)$. The morphism $q_{ij}$ acts on this
boundary stratum by forgetting all matroid polytopes in the tiling except for
those intersecting $\Delta_{\beta_{ij}}(3,6)$. (See
\cite{alexeevModuliWeightedHyperplane2015} for more details.)

\begin{example}
  The boundary divisor $D_{12,34,56}$ corresponds to the following matroid
  tiling of $\Delta(3,6)$.
  \[
    \{x_{12} \leq 1, x_{1234} \leq 2\}, \{x_{34} \leq 1, x_{3456} \leq 2\},
    \{x_{56} \leq 1, x_{1256} \leq 2\}.
  \]
  The second two polytopes do not intersect $\Delta_{\beta_{56}}(3,6)$, so the
  only polytope which survives under the morphism $q_{56} : \oM(3,6) \to
  \oM_{\beta_{56}}(3,6)$ is the first one. We see that the divisor
  $D_{12,34,56}$ is contracted to a line in $\bP^2 \times \bP^2$.
\end{example}
Interpretations of the images as subvarieties of $\bP^2 \times \bP^2$ will be
given below.

\subsection{Blowup construction of $\wM_1(3,6)$}

\subsubsection{Notation} \label{blowupNotation} In any $\bP^2$ with coordinates
$(z_i)$, $\sum z_i = 0$, define the points $P_{ijk} = \{z_i=z_j=z_k\}$ and the
lines $L_{ij} = \{z_i=z_j\}$. In $\bP^2 \times \bP^2$ with coordinates as above,
also define the hypersurfaces
\[
  Q_{ijk} = V\left(\det
    \begin{bmatrix}
      1 & 1 & 1 \\
      x_i & x_j & x_k \\
      y_i & y_j & y_k
    \end{bmatrix}
  \right).
\]
Each $Q_{ijk}$ a divisor of type $(1,1)$ with a unique singular point at
$P_{ijk} \times P_{ijk}$.

Let $M_0 = \oM(3,6)$, $X_0 = \bP^2 \times \bP^2$ and let $\pi_0 = q_{56} : M_0
\to X_0$. By the remarks above, $\pi_0$ is the identity on the interor $M(3,6)$.
For a boundary divisor $D$ in $M_0$, we write $D^0 = \pi_0(D)$, and by the
remarks above we compute the following images.
\begin{enumerate}
\item (Points)
  \[
    D_{ijk,l56}^0 = P_{ijk} \times P_{ijk}.
  \]
\item (Lines)
  \[
    D_{ij,k5,l6}^0 = L_{ij} \times P_{ijk}, \;\;\; D_{ij,k6,l5}^0 = P_{ijk}
    \times L_{ij}, \;\;\; D_{ij,kl,56}^0 = \Delta(L_{ij}).
  \]
\item (Surfaces)
  \begin{align*}
    D_{i5,jkl6}^0 = P_{jkl} \times \bP^2, & \;\;\; D_{i6,jkl5}^0 = \bP^2 \times P_{jkl}, \\
    D_{ij,kl56}^0 = L_{ij} \times L_{ij}, & \;\;\; D_{56,1234}^0 = \Delta(\bP^2).
  \end{align*}
\item (Divisors)
  \[
    D_{ij5,kl6}^0 = L_{kl} \times \bP^2, \;\;\; D_{ij6,kl5}^0 = \bP^2 \times
    L_{kl}, \;\;\; D_{i56,jkl}^0 = Q_{jkl}.
  \]
\end{enumerate}
We will define a sequence of blowups $f_{k+1} : X_{k+1} \to X_{k}$. We
inductively define $D^{k+1}$ to be the dominant transform of $D^k$ in $X_{k+1}$.
(This is the inverse image if $D^k$ is contained in the center; otherwise it is
the strict transform, cf. \cite{liWonderfulCompactificationArrangement2009}.)

\subsubsection{Statement}

\begin{theorem} \label{BlowupConstruction} Let $X_5 \xrightarrow{f_5} \cdots X_1
  \xrightarrow{f_1} X_0$ be the following sequence of blowups.
  \begin{enumerate}
  \item Blowup the eight surfaces $D_{i5,jkl6}^0$, $D_{i6,jkl5}^0$.
  \item Blowup the four surfaces $D_{ijk,l56}^1$.
  \item Blowup the six surfaces $D_{ij,kl56}^2$.
  \item Blowup the surface $D_{56,1234}^3$.
  \item Blowup the 30 surfaces $D_{ij,kl,mn}^4$.
  \end{enumerate}
  Then $X_5 \cong \wM_1(3,6)$ is the small resolution of $\oM(3,6)$ with fiber
  $\bP^1$ over all singular points of $\oM(3,6)$.
\end{theorem}

\begin{remark} \label{generalizedconstructions} By Kapranov's factorization of
  $q_i : \oM_{0,5} \to \bP^2$ as the blowup of $\bP^2$ at four points in general
  position, we see that the first blowup $X_1 \to X_0$ is just the map
  \[
    q_5 \times q_6 : \oM_{0,5} \times \oM_{0,5} \to \bP^2 \times \bP^2.
  \]
  The blowups $X_5 \xrightarrow{f_5} \cdots \xrightarrow{f_2} X_1$ should be
  thought of as the higher-dimensional generalization of Keel's construction of
  $\oM_{0,n}$, and the blowups $X_5 \to \cdots \to X_1 \to X_0$ should be
  thought of as a higher-dimensional generalization of Kapranov's
  construction. \label{generalizedKeel}
\end{remark}

\begin{remark}[Other blowup orders]
  Swapping the order of the first two blowups, one still obtains $\wM_1(3,6)$.
  In this case, the first blowup is along four points $D_{ijk,l56}^0$ in $\bP^2
  \times \bP^2$, while the remaining blowups are along surfaces. We can also
  blowup the $D_{ij,k5,l6}^3$ and $D_{ij,k6,l5}^3$ before blowing up
  $D_{56,1234}^3$, as it turns out these are all disjoint. None of the other
  obvious blowup orders appear to give a small resolution of $\oM(3,6)$; they
  lead to subvarieties being disjoint which should still intersect in any small
  resolution.
\end{remark}

\subsection{Proof of Theorem \ref{BlowupConstruction}}

Let $Z_k$ be the center of the blowup $f_{k+1} : X_{k+1} \to X_k$. Starting with
the morphism $\pi_0 : M_0 \to X_0$, the sequence of blowups induces for each $k$
a unique morphism $\pi_{k+1} : M_{k+1} \to X_{k+1}$ such that the following
diagram commutes, where $g_{k+1}:M_{k+1} \to M_k$ is the blowup of $M_k$ along
the (scheme-theoretic) inverse image $\pi_{k}^{-1}(Z_k)$ \cite[Corollary
7.15]{hartshorneAlgebraicGeometry1977}.
\[
  \begin{tikzcd}
    M_{k+1} \ar[r, "\pi_{k+1}"] \ar[d, "g_{k+1}"] & X_{k+1} \ar[d, "f_{k+1}"] \\
    M_{k} \ar[r, "\pi_{k}"] & X_{k}
  \end{tikzcd}
\]
For a divisor $D$ in $\oM(3,6)$, we continue to denote by $D$ its dominant
transform in $M_k$.

We will show that $M_5 \cong \wM_1(3,6)$ and the morphism $\pi_5 : M_5 \to X_5$
is an isomorphism.

By the discussion of Remark \ref{generalizedconstructions}, it is enough to
start with the morphism $\pi_1 = r_6 \times r_5 : M_1 = \oM(3,6) \to \oM_{0,5}
\times \oM_{0,5} = X_1$.

\subsubsection{Description of the centers} \label{BlowupCenters}

If the irreducible components of the center of a blowup $f : \wX \to X$
intersect transversally (or are disjoint), then by
\cite{liWonderfulCompactificationArrangement2009}, the blowup $f$ is the same as
the iterated blowup along the irreducible centers in any order.

\begin{enumerate}
\item The blowup $f_1 : X_1 \to X_0$ is along eight irreducible centers
  $D_{i5,jkl6}^0 = P_{jkl} \times \bP^2,D_{i6,jkl5}^0 = \bP^2 \times P_{jkl}$.
  Observe these centers intersect transversally, so we can factor $X_1 \to X_0$
  as the blowup first along all $D_{i5,jkl6}^0$, then along the strict
  transforms of all $D_{i6,jkl5}^0$. Then the first blowup is along a disjoint
  union of four smooth irreducible centers $\cong \bP^2$, and the second is
  along a disjoint union of four smooth irreducible centers $\cong Bl_4 \bP^2
  \cong \oM_{0,5}$.
\item The blowup $f_2 : X_2 \to X_1$ is along four disjoint irreducible centers
  $D_{ijk,l56}^1 = D_{l5} \times D_{l6} \subset \oM_{0,5} \times \oM_{0,5}$.
\item The blowup $f_3 : X_3 \to X_2$ is along six irreducible centers
  $D_{ij,kl56}^2$. In $X_1$ we have $D_{ij,kl56}^1 = D_{ij} \times D_{ij}$,
  which intersects two irreducible centers of $f_2 : X_2 \to X_1$ transversally
  in a point each. Thus $D_{ij,kl56}^2 \cong Bl_2(\bP^1 \times \bP^1)$. Observe
  that $D_{ij,kl56}^2$ intersects $D_{kl,ij56}^2$ transversally in a point and
  is disjoint from the other centers. We can therefore factor the blowup $X_3
  \to X_2$ as the blowup first along the $D_{ij,kl56}^2$ for $ij=12,13,23$, and
  then along the strict transforms of the $D_{ij,kl56}^3$ for $ij=14,24,34$.
  Then this first blowup is along a disjoint union of three smooth irreducible
  centers $\cong Bl_2(\bP^1 \times \bP^1)$, while the second is along a disjoint
  union of three smooth irreducible centers $\cong Bl_3(\bP^1 \times \bP^1)
  \cong \oM_{0,5}$.
\item The blowup $f_4 : X_4 \to X_3$ is along the single irreducible center
  $D_{56,1234}^3 \cong \oM_{0,5}$; this is the strict transform of
  $\Delta(\oM_{0,5}) \subset X_1$.
\item The blowup $f_5 : X_5 \to X_4$ is along the disjoint union of 30
  irreducible centers $D_{ij,kl,mn}^4$, each isomorphic to $\bP^1 \times \bP^1$.
  (In $X_1$, we have $D_{ij,k5,l6}^1 \cong D_{ij} \times D_{l6}$,
  $D_{ij,k6,l5}^1 \cong D_{l5} \times D_{ij}$, and $D_{ij,kl,56}^1 \cong
  \Delta(D_{ij})$. The dominant transform of $D_{ij,kl,56}^1$ after blowing up
  $D_{ij,kl56}^2$ looks like $\bP^1 \times \bP^1$. Local computations show that
  the intersections of the $D_{ij,kl,mn}^1$ are removed by the previous
  blowups.)
\end{enumerate}

\begin{corollary}
  $X_5$ is a smooth variety.
\end{corollary}

\begin{proof}
  By the remarks preceding the corollary, we can write $X_5 \xrightarrow{f_5}
  \cdots \xrightarrow{f_2} X_1$ as a sequence of blowups of smooth varieties
  along smooth irreducible centers.
\end{proof}

\begin{lemma} \label{CompleteIntersections}\
  \begin{enumerate}
  \item $D_{ijk,l56}^1$ is the complete intersection of $D_{l5,ijk6}^1$ and
    $D_{l6,ijk5}^1$.
  \item $D_{ij,kl56}^2$ is the complete intersection of $D_{kl5,ij6}^2$ and
    $D_{kl6,ij5}^2$.
  \item $D_{56,1234}^3$ is the intersection of any three $D_{i56,jkl}^3$. It is
    not a complete intersection, but it is a locally complete intersection.
  \item $D_{ij,k5,l6}^4$ is the complete intersection of $D_{kl5,ij6}^4$ and
    $D_{l6,ijk5}^4$.
  \item $D_{ij,k6,l5}^4$ is the complete intersection of $D_{kl6,ij5}^4$ and
    $D_{l5,ijk6}^4$.
  \item $D_{ij,kl,56}^4$ is the complete intersection of $D_{ij,kl56}^4$ and
    $D_{i56,jkl}^4$.
  \end{enumerate}
\end{lemma}

\begin{proof}
  The analogous claims are obvious on $X_0$ (or $X_1$), and follow on $X_k$ by
  direct calculation or \cite[Lemma
  2.9]{liWonderfulCompactificationArrangement2009}.
\end{proof}

\subsubsection{Preimages of centers} \label{preimage_structure}

\begin{lemma} \label{preimageformulas} Let $\pi : M \to X$ be any morphism of
  schemes, let $Z \subset X$ be a regularly embedded closed subscheme, and
  consider the diagram
  \[
    \begin{tikzcd}
      \wM \ar[d,"g"] \ar[r, "\widetilde{\pi}"] & \wX \ar[d, "f"] \\
      M \ar[r, "\pi"] & X
    \end{tikzcd}
  \]
  where $f : \wX \to X$ is the blowup of $X$ along $Z$ and $g : \wM \to M$ is
  the blowup of $M$ along the scheme-theoretic preimage $\pi^{-1}Z$.

  Let $Y$ be a closed subscheme of $X$. Let $\wY$ be its dominant transform in
  $\wX$.
  \begin{enumerate}
  \item If $Y \subset Z$ or $Y$ intersects $Z$ transversally, then
    $\widetilde{\pi}^{-1}\wY = g^{-1}\pi^{-1}Y$.
  \item If $Y$ contains $Z$, then $\widetilde{\pi}^{-1}Y$ is the residual scheme
    to the exceptional divisor of $\wM$ in $g^{-1}\pi^{-1}Y$.
  \end{enumerate}
\end{lemma}

\begin{proof}
  \begin{enumerate}
  \item If $Y \subset Z$, then $\wY = f^{-1}Y$ by definition, while if $Y$
    intersects $Z$ transversally, then the result $\wY = f^{-1}Y$ is standard
    \cite[B.6]{fultonIntersectionTheory1998}. By commutativity,
    \[
      \widetilde{\pi}^{-1}\wY = \widetilde{\pi}^{-1}f^{-1}Y = g^{-1}\pi^{-1}Y.
    \]
  \item If $Y$ contains $Z$, then $\wY$ is the residual scheme to $E$ in
    $f^{-1}Y$ \cite[B.6.10]{fultonIntersectionTheory1998}. Then
    $\widetilde{\pi}^{-1}\wY$ is the residual scheme to $\widetilde{\pi}^{-1}E$
    in $\widetilde{\pi}^{-1}f^{-1}Y$. The result follows by commutativity.
  \end{enumerate}
\end{proof}

In the computations below we only consider the set-theoretic inverse images
$\pi_k^{-1}(Z_k)$ of the centers of the blowups. This is enough because, as we
will see, $\pi_k^{-1}(Z_k)$ is Cartier except possibly near the singular points.
We can describe $\pi_k$ explicitly near these points (coordinates as in Section
\ref{singularities}):
\begin{enumerate}
\item near a point like $P_{ij,k5,l6}$ or $P_{ij,k6,l5}$, $\pi_1$ is given by
  $(z_1,z_2,z_3,z_4,z_5,z_6) \mapsto (z_1,z_2,z_5,z_6)$;
\item near a point like $P_{ij,kl,56}$, $\pi_1$ is given by
  $(z_1,z_2,z_3,z_4,z_5,z_6) \mapsto (z_1,z_2z_4,z_3z_5,z_6)$.
\end{enumerate}
If $M_k \to M_{k-1}$ is not a small resolution at a singular point, then $\pi_k$
and $\pi_{k-1}$ have the same description at this point. It follows from this
description that $\pi_k^{-1}(Z_k)$ is reduced near the singular points of $M_k$.

\subsubsection{Step 1}

\begin{lemma} \label{Preimages1} (Recall the pullback formulas from Proposition
  \ref{PullbackFormulas}.)
  \begin{enumerate}
  \item $\pi_1^{-1}(D_{l5,ijk6}^1) = r_6^*(D_{l5})$.
  \item $\pi_1^{-1}(D_{l6,ijk5}^1) = r_5^*(D_{l6})$.
  \item $\pi_1^{-1}(D_{kl5,ij6}^1) = r_6^*(D_{ij})$.
  \item $\pi_1^{-1}(D_{kl6,ij5}^1) = r_5^*(D_{ij})$.
  \item $\pi_1^{-1}(D_{i56,jkl}^1) = f_i^*(D_{56} + D_{jk} + D_{jl} + D_{kl})$.
  \end{enumerate}
\end{lemma}

\begin{proof}
  If $p_1,p_2 : X_1 \to \oM_{0,5}$ denote the two natural projections, note that
  $D_{l5,ijk6}^1 = p_1^*(D_{l5})$, $D_{l6,ijk5}^1 = p_2^*(D_{l6})$,
  $D_{kl5,ij6}^1 = p_1^*(D_{ij})$, and $D_{kl6,ij5}^1 = p_2^*(D_{ij})$, so the
  result is clear in these cases from Proposition \ref{PullbackFormulas}. The
  remaining case $\pi_1^{-1}(D_{i56,jkl}^1)$ is verified explicitly, similar to
  the proof of Proposition \ref{PullbackFormulas}.
\end{proof}

\begin{corollary}
  $\pi_1^{-1}(D_{ijk,l56}^1) = D_{ijk,l56}$.
\end{corollary}

\begin{proof}
  By Lemma \ref{CompleteIntersections}, $D_{ijk,l56}^1 = D_{l5,ijk6}^1 \cap
  D_{l6,ijk5}^1$, so (set-theoretically)
  \begin{align*}
    \pi_1^{-1}(D_{ijk,l56}^1) &= \pi_1^{-1}(D_{l5,ijk6}^1) \cap \pi_1^{-1}(D_{l6,ijk5}^1) \\
                              &= D_{ijk,l56} \text{ by Lemma \ref{Preimages1}}.
  \end{align*}
\end{proof}

\subsubsection{Step 2}

By the corollary, $\pi_1^{-1}(Z_1) = \bigcup D_{ijk,l56}$, which is disjoint
from all the singular points of $M_1$ and is Cartier, so $g_2 : M_2 \to M_1$ is
an isomorphism.

\begin{lemma} \label{Preimages2}
  \begin{enumerate}
  \item $\pi_2^{-1}(D_{l5,ijk6}^2) = r_6^*(D_{l5}) - D_{ijk,l56}$.
  \item $\pi_2^{-1}(D_{l6,ijk5}^2) = r_5^*(D_{l6}) - D_{ijk,l56}$.
  \item $\pi_2^{-1}(D_{kl5,ij6}^2) = r_6^*(D_{ij})$.
  \item $\pi_2^{-1}(D_{kl6,ij5}^2) = r_5^*(D_{ij})$.
  \item $\pi_2^{-1}(D_{i56,jkl}^2) = f_i^*(D_{56} + D_{jk} + D_{jl} + D_{kl}) -
    D_{ijk,l56} - D_{ijl,k56} - D_{ikl,j56}$.
  \end{enumerate}
\end{lemma}

\begin{proof}
  This follows from Lemma \ref{Preimages1} and Lemma \ref{preimageformulas}.
\end{proof}

\begin{corollary}
  $\pi_2^{-1}(D_{ij,kl56}^2) = D_{ij,kl56} \cup D_{ij,kl,56}$.
\end{corollary}

\begin{proof}
  By Lemma \ref{CompleteIntersections}, $D_{ij,kl56}^2 = D_{kl5,ij6}^2 \cap
  D_{kl6,ij5}^2$, so (set-theoretically)
  \begin{align*}
    \pi_2^{-1}(D_{ij,kl56}^2) &= \pi_2^{-1}(D_{kl5,ij6}^2) \cap \pi_2^{-1}(D_{kl6,ij5}^2) \\
                              &= D_{ij,kl56} \cup D_{ij,kl,56} \text{ by Lemma \ref{Preimages2}.}
  \end{align*}
\end{proof}

\subsubsection{Step 3}

By the corollary, $\pi_2^{-1}(Z_2) = \bigcup (D_{ij,kl56} \cup D_{ij,kl,56})$.
Since this is Cartier away from the singular points, we only need to look at
$\pi_2^{-1}(Z_2)$ near the singular points. By Section \ref{preimage_structure},
$\pi_2^{-1}(Z_2)$ is reduced near these points.

Up to symmetry, there are two cases.
\begin{enumerate}
\item Near $P_{ij,kl,56}$, the divisor $\pi_2^{-1}(Z_2)$ looks like
  \[
    D_{ij,kl56} + D_{ij,kl,56} + D_{kl,ij56} + D_{kl,ij,56}.
  \]
  By Proposition \ref{BoundaryDivs}, this is Cartier at $P_{ij,kl,56}$, so the
  blowup does nothing near this point.
\item Near $P_{ij,k5,l6}$, the divisor $\pi_2^{-1}(Z_2)$ looks like
  $D_{ij,kl56}$. By Lemma \ref{LocalSmallResolutions}, the blowup gives the
  small resolution with fiber $\bP^1$ over this point.
\end{enumerate}

We see that $M_3$ is a partial desingularization of $M_2 = \oM(3,6)$: near the
points of the form $P_{ij,k5,l6}$ it is the small resolution with fiber $\bP^1$,
and near the remaining singular points $P_{ij,kl,56}$ it does not change.

\begin{lemma} \label{Preimages3}
  \begin{enumerate}
  \item $\pi_3^{-1}(D_{l5,ijk6}^3) = r_6^*(D_{l5}) - D_{ijk,l56}$.
  \item $\pi_3^{-1}(D_{l6,ijk5}^3) = r_5^*(D_{l6}) - D_{ijk,l56}$.
  \item $\pi_3^{-1}(D_{kl5,ij6}^3) = D_{kl5,ij6} + D_{ij,k5,l6} + D_{ij,l5,k6}$.
  \item $\pi_3^{-1}(D_{kl6,ij5}^3) = D_{kl6,ij5} + D_{ij,k6,l5} + D_{ij,l6,5}$.
  \item $\pi_3^{-1}(D_{i56,jkl}^3) = D_{i56,jkl} + D_{56,1234} + D_{ij,kl,56} +
    D_{ik,jl,56} + D_{il,jk,56} + D_{jk,i5,l6} + D_{jk,i6,l5} + D_{jl,i5,k6} +
    D_{jl,i6,k5} + D_{kl,i5,j6} + D_{kl,i6,j5}$.
  \item $\pi_3^{-1}(D_{ij,kl56}^3) = D_{ij,kl56} + D_{ij,kl,56}$.
  \end{enumerate}
\end{lemma}

\begin{proof}
  This follows from Lemma \ref{Preimages2} and Lemma \ref{preimageformulas}.
\end{proof}

\begin{corollary}
  $\pi_3^{-1}(D_{56,1234}^3) = D_{56,1234}$.
\end{corollary}

\begin{proof}
  By Lemma \ref{CompleteIntersections}, we can write (set-theoretically)
  \begin{align*}
    \pi_3^{-1}(D_{56,1234}^3) &= \pi_3^{-1}(D_{i56,jkl}^3) \cap \pi_3^{-1}(D_{j56,ikl}^3) \cap \pi_3^{-1}(D_{k56,ijl}^3) \\
                              &= D_{56,1234} \text{ by Lemma \ref{Preimages3}.}
  \end{align*}
\end{proof}

\subsubsection{Step 4}

By the corollary, $\pi_3^{-1}(Z_3) = D_{56,1234}$, which fails to be Cartier at
the three remaining singular points $P_{ij,kl,56}$ of $M_3$. By Section
\ref{preimage_structure}, $\pi_3^{-1}(Z_3)$ is reduced near these points, and by
Lemma \ref{LocalSmallResolutions} the blowup gives the small resolution with
fiber $\bP^1$ over these points. Combined with our description of $M_3$, we see
that $M_4 \cong \wM_1(3,6)$.

\begin{lemma} \label{Preimages4}
  \begin{enumerate}
  \item $\pi_4^{-1}(D_{l5,ijk6}^4) = r_6^*(D_{l5}) - D_{ijk,l56}$.
  \item $\pi_4^{-1}(D_{l6,ijk5}^4) = r_5^*(D_{l6}) - D_{ijk,l56}$.
  \item $\pi_4^{-1}(D_{kl5,ij6}^4) = D_{kl5,ij6} + D_{ij,k5,l6} + D_{ij,l5,k6}$.
  \item $\pi_4^{-1}(D_{kl6,ij5}^4) = D_{kl6,ij5} + D_{ij,k6,l5} + D_{ij,l6,5}$.
  \item $\pi_4^{-1}(D_{i56,jkl}^4) = D_{i56,jkl} + D_{ij,kl,56} + D_{ik,jl,56} +
    D_{il,jk,56} + D_{jk,i5,l6} + D_{jk,i6,l5} + D_{jl,i5,k6} + D_{jl,i6,k5} +
    D_{kl,i5,j6} + D_{kl,i6,j5}$.
  \item $\pi_4^{-1}(D_{ij,kl56}^4) = D_{ij,kl56} + D_{ij,kl,56}$.
  \end{enumerate}
\end{lemma}

\begin{proof}
  This follows from Lemma \ref{Preimages3} and Lemma \ref{preimageformulas}.
\end{proof}

\begin{corollary}
  $\pi_4^{-1}(D_{ij,kl,mn}^4) = D_{ij,kl,mn}$.
\end{corollary}

\begin{proof}
  Up to symmetry, there are two cases.
  \begin{enumerate}
  \item Since $D_{ij,k5,l6}^4 = D_{kl5,ij6}^4 \cap D_{l6,ijk5}^4$ by Lemma
    \ref{CompleteIntersections}, we have (set-theoretically)
    \begin{align*}
      \pi_4^{-1}(D_{ij,k5,l6}^4) &= \pi_4^{-1}(D_{kl5,ij6}^4) \cap \pi_4^{-1}(D_{l6,ijk5}^4) \\
                                 &= D_{ij,k5,l6} \text{ by Lemma \ref{Preimages4}.}
    \end{align*}
  \item Since $D_{ij,kl,56}^4 = D_{ij,kl56}^4 \cap D_{i56,jkl}^4$ by Lemma
    \ref{CompleteIntersections}, we have (set-theoretically)
    \begin{align*}
      \pi_4^{-1}(D_{ij,kl,56}^4) &= \pi_4^{-1}(D_{ij,kl56}^4) \cap \pi_4^{-1}(D_{i56,jkl}^4) \\
                                 &= D_{ij,kl,56} \text{ by Lemma \ref{Preimages4}.}
    \end{align*}
  \end{enumerate}
\end{proof}

\subsubsection{Step 5}

By the corollary, $\pi_4^{-1}(Z_4) = \bigcup D_{ij,kl,mn}$, which is a Cartier
divisor on $M_4$, so $M_5 \cong M_4 \cong \wM_1(3,6)$.

\begin{claim}
  The morphism $\pi_5 : M_5 \to X_5$ is an isomorphism.
\end{claim}

\begin{proof}
  Since the original morphism $\pi_1 : \oM(3,6) \to X_1$ is birational, and each
  intermediate morphism is a blowup, it follows that $\pi_5: M_5 \to X_5$ is
  birational.

  The exceptional divisors of the composition $M_5 \to M_1 \xrightarrow{\pi_1}
  X_1$ are exactly $D_{ijk,l56},D_{ij,kl56},D_{56,1234}$, and $D_{ij,kl,mn}$,
  and by construction these are also divisors in $X_5$. Thus $\pi_5: M_5 \to
  X_5$ has no exceptional divisors. Since $\pi_5$ is a projective birational
  morphism between smooth varieties, it follows that $\pi_5$ is an isomorphism.
\end{proof}

This completes the proof of Theorem \ref{BlowupConstruction}.

\subsection{Independence from Luxton's methods} \label{IndependentLuxton}

We show here that our construction is independent from Luxton's tropical methods
on the structure of $\oM(3,6)$. This is important to ensure that our proof that
$(\oM(3,6),B)$ is log canonical (Section \ref{BirationalGeometry}) is not
circular. The facts about $\oM(3,6)$ used in the proof of Theorem
\ref{BlowupConstruction} are as follows.
\begin{enumerate}
\item The description of the boundary divisors (Section \ref{BoundaryDivisors}).
  Luxton's proof of these descriptions is independent of his tropical arguments;
  he uses only the explicit descriptions of the stable hyperplane arrangements
  parameterized by these boundary divisors \cite[Section
  4.2.4]{luxtonLogCanonicalCompactification2008}.
\item The description of the boundary complex (Section \ref{BoundaryComplex}).
  As remarked previously, the boundary complex of $\oM(3,6)$ can be determined
  explicitly by examining the stable hyperplane arrangements parameterized by
  $\oM(3,6)$.
\item The description of the singularities of $\oM(3,6)$ (Proposition
  \ref{singularitiesDescription}). We give an independent proof below.
\end{enumerate}

\begin{proof}[Proof of Proposition \ref{singularitiesDescription}]
  Let
  \[
    q = q_{12} \times q_{34} \times q_{56} : \oM(3,6) \to (\bP^2 \times
    \bP^2)^3,
  \]
  and let $Z = q(\oM(3,6))$. Then $q$ defines a birational morphism $\oM(3,6)
  \to Z$ which is an isomorphism on the interior. As in Section \ref{birmaps}
  and the setup preceding the statement of Theorem \ref{BlowupConstructions}, we
  can explicitly compute the behavior of $q$ on each boundary stratum. We find
  that the exceptional locus of $q$ consists of the following.
  \begin{enumerate}
  \item Each of the twelve divisors $D_{ij,kl,mn}$ with exactly one of
    $ij,kl,mn$ equal to $12,34$, or $56$ gets contracted to a surface $\cong
    \bP^1 \times \bP^1$ in $Z$.
  \item Each of the six surfaces $D_{i5,jkl6} \cap D_{j6,ikl5}$ ($i \neq j$)
    gets contracted to a singular point $q(P_{i5,j6,kl}) \in Z$.
  \item Sixteen lines of the form $D_{ijk,lmn} \cap D_{ilm,jkn} \cap
    D_{jmn,ikl}$ are contracted to points. These lines connect the divisors
    $D_{12,34,56}$ and $D_{12,56,34}$ to the divisors $D_{ij,kl,mn}$ where none
    of $ij,kl,mn$ are $12,34$ or $56$.
  \end{enumerate}

  Similarly, as in Section \ref{birmaps} we determine explicit equations for $q$
  on the interior, hence equations for $Z \subset (\bP^2 \times \bP^2)^3$. From
  this we find that $Z$ is singular at a total of 31 points:
  \begin{enumerate}
  \item 15 points $q(P_{ij,kl,mn})$, and
  \item 16 additional points, the images of the contracted lines.
  \end{enumerate}
  Furthermore, we compute that each singular point of $Z$ looks like $0 \in
  C(\bP^1 \times \bP^2)$.

  Since $q$ does not contract anything in a neighborhood $U$ of the point
  $P_{12,34,56}$, it follows that $q$ is an isomorphism on $U$, hence by the
  description of the singularites of $Z$ we find that $P_{12,34,56} \in
  \oM(3,6)$ looks like $0 \in C(\bP^1 \times \bP^2)$. By symmetry we deduce that
  all 15 points $P_{ij,kl,mn} \in \oM(3,6)$ have the same form.

  Near the extra 16 singular points of $Z$, the morphism $q : \oM(3,6) \to Z$
  looks like the resolution of $0 \in C(\bP^1 \times \bP^2)$ with fiber $\bP^1$.
  We conclude that $\oM(3,6)$ is smooth along these lines. It follows that
  $\oM(3,6)$ is smooth everywhere except for the 15 points $P_{ij,kl,mn}$, each
  of which looks like $0 \in C(\bP^1 \times \bP^2)$.
\end{proof}

\section{Intersection theory of small
  resolutions} \label{IntersectionTheorySmall}

\begin{theorem} \label{ChowResolutions} Let $\wM_{S_1,S_2}(3,6)$ be any small
  resolution of $\oM(3,6)$.
  \begin{enumerate}
  \item
    \[
      A^*(\widetilde{M}_{S_1,S_2}(3,6)) =
      \frac{\bZ[D_{ijk,lmn},D_{ij,klmn},D_{ij,kl,mn}]}{\text{the following
          relations}}
    \]
    \begin{enumerate}
    \item (Linear relations)
      \begin{enumerate}
      \item $D_{ij,kl,mn}=D_{mn,ij,kl}=D_{kl,mn,ij}$.
      \item $f^*(0) = f^*(1) = f^*(\infty)$, where $f$ is any composition of
        restriction and forgetful maps $\wM_{S_1,S_2}(3,6) \to \oM(3,6)
        \xrightarrow{r_i} \oM_{0,5} \xrightarrow{f_j} \oM_{0,4} = \bP^1$.
      \end{enumerate}
    \item (Multiplicative relations) $\prod D_i = 0$ if $\bigcap D_i =
      \emptyset$ in $\widetilde{M}_{S_1,S_2}(3,6)$ (see Remark
      \ref{BoundaryIntersections} below).
    \end{enumerate}
  \item The nontrivial (i.e. $\neq 0,1$) ranks of the Chow groups are
    \begin{align*}
      \rk A^1(\widetilde{M}_{S_1,S_2}(3,6)) &= 51, \\
      \rk A^2(\widetilde{M}_{S_1,S_2}(3,6)) &= 127 + \lvert S_2 \rvert, \\
      \rk A^3(\widetilde{M}_{S_1,S_2}(3,6)) &= 51.
    \end{align*}
  \item
    \begin{enumerate}
    \item $\Pic \wM_{S_1,S_2}(3,6)$ is generated by the boundary divisors,
      modulo the linear relations.
    \item A basis for $\Pic \wM_{S_1,S_2}(3,6)$ is given by
      \begin{enumerate}
      \item
        $D_{156,234},D_{256,134},D_{345,126},D_{346,125},D_{356,124},D_{456,123}$,
      \item all 15 $D_{ij,klmn}$,
      \item all 30 $D_{ij,kl,mn}$.
      \end{enumerate}
    \end{enumerate}
  \item (Over $\bC$) The map $cl : A_*(\widetilde{M}_{S_1,S_2}(3,6)) \to
    H_*(\widetilde{M}_{S_1,S_2}(3,6))$ is an isomorphism.
  \end{enumerate}
\end{theorem}

\begin{remark}
  The relations $D_{ij,kl,mn}=D_{mn,ij,kl}=D_{kl,mn,ij}$ reflect that these
  divisors are all the same. We will assume these relations implicitly for the
  remainder of this paper.
\end{remark}

\begin{remark} \label{BoundaryIntersections} Recall that the intersections of
  boundary divisors on $\oM(3,6)$ or any of its small resolutions are described
  by the boundary complex, see Sections \ref{BoundaryComplex},\ref{resolutions}.
  Explicitly, the multiplicative relations on any $A^*(\wM_{S_1,S_2}(3,6))$ are
  as follows.
  \begin{enumerate}
  \item (Relations from $\oM(3,6)$)
    \begin{enumerate}
    \item $D_{ijk,lmn}D_{abc,def} = 0$ if $\lvert ijk \cap abc \rvert = 2$.
    \item $D_{ij,klmn}D_{ab,cdef} = 0$ if $\lvert ij \cap ab \rvert = 1$.
    \item $D_{ij,kl,mn}D_{ab,cd,ef} = 0$ unless $\{ij,kl,mn\} = \{ab,cd,ef\}$.
    \item $D_{ijk,lmn}D_{ab,cdef} = 0$ if $\lvert ijk \cap ab \rvert = 1$.
    \item $D_{ijk,lmn}D_{ab,cd,ef} = 0$ unless $ijk=abc$ or $ijk=abd$ (after
      sufficient cyclic permutation).
    \item $D_{ij,klmn}D_{ab,cd,ef} = 0$ unless $ij=ab$ or $cd$ or $ef$.
    \end{enumerate}
  \item (Relations from $S_1$) $D_{ij,kl,mn}D_{ij,mn,kl} = 0$ for $P_{ij,kl,mn}
    \in S_1$.
  \item (Relations from $S_2$) $D_{ij,klmn}D_{kl,ijmn}D_{mn,ijkl} = 0$ for
    $P_{ij,kl,mn} \in S_2$.
  \end{enumerate}
\end{remark}

\subsection{General results} \label{GeneralResults}

We begin by recalling some general results on Chow rings and blowups.

\subsubsection{Setup}
Let $X$ be a nonsingular variety, let $Z \subset X$ be a regularly embedded
closed subscheme of codimension $d > 1$, and let $f:\wX \to X$ be the blowup of
$X$ along $Z$, with exceptional divisor $E$. By \cite[Example
17.5.1(c)]{fultonIntersectionTheory1998}, there is a split exact sequence
\begin{equation} \label{ChowRingSequence} 0 \to A^{k-d}Z \to A^{k-1}E \oplus
  A^kX \to A^k\wX \to 0.
\end{equation}

\subsubsection{Decomposition of Chow groups}
\begin{lemma} \label{ChowGroupsofBlowup} $A^k\wX \cong A^kX \oplus
  \bigoplus_{i=1}^{d-1} A^{k-i}Z$.
\end{lemma}

\begin{proof}
  By \cite[Example 17.5.1(b)]{fultonIntersectionTheory1998}, $A^{k-1}E \cong
  \bigoplus_{i=1}^{d}A^{k-i}Z$, so the exact sequence \eqref{ChowRingSequence}
  becomes
  \[
    0 \to A^{k-d}Z \to A^{k-d}Z \oplus \bigoplus_{i=1}^{d-1}A^{k-i}Z \oplus A^kX
    \to A^k\wX \to 0.
  \]
\end{proof}

\subsubsection{Ranks of Chow groups}
\begin{corollary} \label{ChowGroupRanks} $\rk A^k\wX = \rk A^kX +
  \sum_{i=1}^{d-1} \rk A^{k-i}Z$.
\end{corollary}

\begin{proof}
  Immediate from Lemma \ref{ChowGroupsofBlowup}.
\end{proof}

\subsubsection{Homological results}
The next result is over $\bC$. Recall from
\cite{keelIntersectionTheoryModuli1992} that a scheme $X$ is called an
\textit{HI scheme} if $cl : A_*X \to H_*X$ is an isomorphism.

\begin{lemma} \label{HIschemes}\
  \begin{enumerate}
  \item If $X$ and $Z$ are HI schemes, then so is $\wX$.
  \item If $\wX$ and $Z$ are HI schemes, then so is $X$.
  \end{enumerate}
\end{lemma}

\begin{proof}
  \begin{enumerate}
  \item The first part is \cite[Theorem A.2]{keelIntersectionTheoryModuli1992}.
   
  \item Consider the following diagram with exact rows \cite[Proof of Theorem
    A.2]{keelIntersectionTheoryModuli1992}.
    \[
      \begin{tikzcd}
        0 \ar[r] & A_*Z \ar[r]\ar[d,"cl"] & A_*E \oplus A_*X \ar[r] \ar[d,
        "cl\oplus cl"] & A_*\wX \ar[r] \ar[d,"cl"] & 0 \\
        0 \ar[r] & H_*Z \ar[r] & H_*E \oplus H_*X \ar[r] & H_*\wX \ar[r] & 0
      \end{tikzcd}
    \]
    Since $cl : A_*Z \to H_*Z$ and $cl : A_*X \to H_*\wX$ are isomorphisms, it
    follows by the short five lemma that
    \[
      cl \oplus cl : A_*E \oplus A_*X \to H_*E \oplus H_*X
    \]
    is an isomorphism. The first direct summand $cl : A_*E \to H_*E$ is already
    an isomorphism because $E$ is a projective bundle over $Z$. It follows that
    $cl : A_*X \to H_*X$ is an isomorphism.
  \end{enumerate}
\end{proof}

\subsubsection{Generators of Chow rings}
\begin{lemma} \label{ChowGenerators}\ Suppose $A^*X$ is generated by divisor
  classes $D_1,\ldots,D_n$. Also assume that $A^*Z$ is generated by $A^1Z$.
  \begin{enumerate}
  \item $A^*\wX$ is generated by $f^*D_1,\ldots,f^*D_n,E$.
  \item $A^*\wX$ is generated by $\wD_1,\ldots,\wD_n,E$, where $\wD_i$ is the
    strict transform of $D_i$.
  \item Suppose $\wX$ is also the blowup of another nonsingular variety $Y$,
    with the same exceptional divisor $E$. Let $g : \wX \to Y$ denote this
    blowup. Then $A^*Y$ is generated by divisor classes $D_1',\ldots,D_n'$ such
    that $\wD_i' = \wD_i$.
  \end{enumerate}
\end{lemma}

\begin{proof}
  Since $A^*X$ and $A^*Z$ are both generated in degree 1, it follows by Lemma
  \ref{ChowGroupsofBlowup} that $A^*\wX$ is also generated in degree 1. The
  short exact sequence \eqref{ChowRingSequence} implies
  \[
    A^1\wX \cong A^1X \oplus A^0E,
  \]
  so the first part follows.

  The second part follows from the first because by \cite[Theorem
  6.7]{fultonIntersectionTheory1998}, we can write $f^*D_i$ as a sum of $\wD_i$
  and some multiple of the exceptional divisor.

  Let $D_1',\ldots,D_n' \in A^1Y$ be divisor classes such that $\wD_i' = \wD_i$.
  By \cite[Theorem 6.7]{fultonIntersectionTheory1998}, write
  \[
    \wD_i = g^*D_i' - m'E.
  \]
  Then since $\wD_1,\ldots,\wD_n,E$ generate $A^*\wX$, it follows that
  $g^*D_1',\ldots,g^*D_n',E$ generate $A^*\wX$. From the isomorphism
  \[
    A^1\wX \cong A^1Y \oplus A^0E,
  \]
  we find that $A^1Y$ is generated by $D_1',\ldots,D_n'$. It remains to show
  that $A^*Y$ is generated by $A^1Y$.

  Let $\alpha \in A^kY$ be nonzero. Then $g^*\alpha \in A^k\wX$, so we can write
  \[
    g^*\alpha = \beta_1 \cdots \beta_k,
  \]
  where $\beta_i \in A^1\wX$ is a linear combination of $\wD_1,\ldots,\wD_n,E$.
  By \cite[Proposition 6.7(b)]{fultonIntersectionTheory1998}, $g_*g^*\alpha =
  \alpha$, hence
  \[
    \alpha = g_*\beta_1 \cdots g_*\beta_k.
  \]
  By definition each $g_*\beta_i$ is either zero or in $A^1Y$. Since $\alpha
  \neq 0$, we conclude that all $g_*\beta_i \in A^1Y$, so $A^*Y$ is generated by
  $A^1Y$.
\end{proof}

\subsubsection{Keel's formula}

Our main technical tool for computing the Chow ring of a blowup is Keel's
formula.

\begin{lemma}[Keel's formula {\cite[Theorem A.1]{keelIntersectionTheoryModuli1992}}]
  Suppose $A^*X \to A^*Z$ is surjective with kernel $J_{Z/X}$. Let $P_{Z/X}(t)
  \in A^*(X)[t]$ be a Chern polynomial for $Z$ in $X$, i.e. any polynomial in
  $A^*(X)[t]$ which restricts to the Chern polynomial for $N_{Z/X}$ in
  $A^*(Z)[t]$. Then
  \[
    A^*(\wX) = \frac{A^*(X)[E]}{J_{Z/X} \cdot E, P_{Z/X}(-E)}.
  \]
\end{lemma}

\subsubsection{Chern polynomial relations}

The next two sections gather some results which will be used for applying Keel's
formula to determine the Chow ring of an iterated blowup. Most of these results
originate in \cite[Section 5]{fultonCompactificationConfigurationSpaces1994}.

\begin{lemma}\label{Chernformulas}\
  \begin{enumerate}
  \item A Chern polynomial for a complete intersection of divisors with classes
    $D_i$ is $\prod (t + D_i)$.
  \item Suppose $Y \subset X$ is a regularly embedded closed subscheme with
    Chern polynomial $P_{Y/X}(t)$.
    \begin{enumerate}
    \item If $Y$ intersects $Z$ transversally then $P_{Y/X}(t)$ is a Chern
      polynomial for the strict transform $\wY$ of $Y$ in $\wX$.
    \item If $Y$ contains $Z$ then $P_{Y/X}(t-E)$ is a Chern polynomial for
      $\wY$ in $\wX$.
    \end{enumerate}
  \end{enumerate}
\end{lemma}

\begin{proof}
  See \cite[Section 5]{fultonCompactificationConfigurationSpaces1994}.
\end{proof}

The following formula for the Chern class of the normal bundle of a strict
transform will also be useful. More general formulas can be found in
\cite{aluffiChernClassesBlowups2010a}.

\begin{lemma} \label{Chernbad}\ Let $Y$ be a regularly embedded closed subscheme
  of $X$ which intersects $Z$ transversally inside of some regularly embedded
  closed subscheme $W \subset X$. Let $\wY$ and $\wW$ denote the strict
  transforms of $Y$ and $W$, respectively. Then
  \[
    c(N_{\wY/\wX}) = \frac{\pi^*c(N_{Y/X})c((\pi^*N_{W/X} \otimes
      \cO(-E))\vert_{\wY})}{\pi^*c(N_{W/X}\vert_{\wY})}.
  \]
\end{lemma}

\begin{proof}
  See \cite[Section 4]{aluffiChernClassesBlowups2010a}. This follows from the
  standard normal bundle exact sequences for the inclusions $Y \subset W \subset
  X$ and $\wY \subset \wW \subset \wX$, together with the usual formulas for the
  normal bundle of a strict transform of a subvariety which contains the center
  or intersects it transversally.
\end{proof}

\subsubsection{Restriction relations}

Let $Y \subset X$ be a regularly embedded closed subscheme. Assume $A^*(X) \to
A^*(Y)$ is surjective.

\begin{lemma}\label{restrictionformulas}\
  \begin{enumerate}
  \item If $Y$ intersects $Z$ transversally, then $A^*(\wX) \to A^*(\wY)$ is
    surjective with kernel $\langle J_{Y/X}, J_{Y \cap Z/X} \cdot E \rangle$.
  \item If $Y$ intersects $Z$ in a Cartier divisor $V$ on $Y$, then $A^*(\wX)
    \to A^*(\wY)$ is surjective with kernel $\langle J_{Y/X}, E - \alpha
    \rangle$, where $\alpha \in A^1(X)$ is any class which restricts to the
    class of $V$ on $A^*(Y)$.
  \item If $Z$ is the transversal intersection of $Y$ and another regularly
    embedded closed subscheme $W \subset X$ with $A^*(X) \to A^*(W)$ surjective,
    then $A^*(\wX) \to A^*(\wY)$ is surjective with kernel $\langle
    J_{Y/X},P_{W/Y}(-E) \rangle$.
  \item Suppose $Y \subset Z$ and $A^*(Z) \to A^*(Y)$ is surjective with kernel
    $J_YZ$. Then $A^*(\wX) \to A^*(\wY)$ is surjective with kernel
    $(i^*)^{-1}(J_{Y/Z})$, where $i^* : A^*(\wX) \to A^*(E)$ is the restriction
    to the exceptional divisor.
  \end{enumerate}
\end{lemma}

\begin{proof}
  \begin{enumerate}
  \item This is \cite[Lemma 2.3]{petersenChowRingFultonMacPherson2017}, which is
    a corrected version of \cite[Lemma
    5.4]{fultonCompactificationConfigurationSpaces1994}.
  \item Surjectivity is immediate, since $\wY \cong Y$, the restriction map
    $A^*(X) \to A^*(Y)$ is surjective, and $A^*(X)$ is a subring of $A^*(\wX)$.
    The claim about the kernel is \cite[Remark, page
    566]{keelIntersectionTheoryModuli1992}.
  \item \cite[Lemma 5.5]{fultonCompactificationConfigurationSpaces1994}.
  \item Note that $E = \bP(N_{Z/X})$ and $\wY = \bP(N_{Z/X}\vert_Y)$. The
    restriction of the Chern polynomial for $N_{Z/X}$ (in $A^*(Z)[t]$) is
    precisely the Chern polynomial for $N_{Z/X}\vert_Y$ (in $A^*(Y)[t]$). It
    follows by the standard formulas for the Chow ring of a projective bundle
    that $A^*(E) \to A^*(\wY)$ is surjective with kernel $J_{Y/Z}$. The
    restriction map $i^*:A^*(\wX) \to A^*(E)$ is also surjective (e.g. by Keel's
    formula), hence $A^*(\wX) \to A^*(\wY)$ is surjective with kernel
    $(i^*)^{-1}(J_{Y/Z})$.
  \end{enumerate}
\end{proof}

% \subsubsection{A general result}
%
% The following lemma will be used for showing that the Chow rings of the small
% resolutions have the desired presentations.
%
% \begin{lemma} \label{ChowRingLemma} Let $\varphi : R \to S$ be a surjective
%   homomorphism of graded rings. Suppose that for all $k$,
%   \begin{enumerate}
%   \item both $R^k$ and $S^k$ are torsion-free.
%   \item $\rk R^k = \rk S^k < \infty$.
%   \end{enumerate}
%   Then $\varphi$ is an isomorphism.
% \end{lemma}
%
%\begin{proof}
%  The hypotheses imply that $\varphi^k : R^k \to S^k$ is a surjective morphism
%  between finitely generated free abelian groups of the same rank, hence
%  $\varphi^k$ is an isomorphism for all $k$, thus $\varphi : R \to S$ is itself
%  an isomorphism.
%\end{proof}

\subsection{Reduction to a single small resolution} \label{ChowRemaining}

\begin{proposition} \label{IterativeProposition} The results of Theorem
  \ref{ChowResolutions} hold for a given small resolution of $\oM(3,6)$ if and
  only if they hold for all small resolutions of $\oM(3,6)$.
\end{proposition}

The proof of this proposition will take the remainder of this subsection.

To prove the proposition, it is enough to show that the results of the theorem
hold for $\wM_{S_1,S_2}(3,6)$ if and only if they hold for
$\wM_{S_1',S_2'}(3,6)$, where $S_1' = S_1 \setminus \{P_{ij,kl,mn}\}$, $S_2' =
S_2 \cup \{P_{ij,kl,mn}\}$ (or vice-versa). Also, by symmetry, we can assume for
concreteness that $P_{ij,kl,mn} = P_{12,34,56}$. To that end we will fix for the
remainder of this section a small resolution $\wM_{S_1,S_2}(3,6)$ with
$P_{12,34,56} \in S_1$, and $S_1',S_2'$ as described.

Let $\wM(3,6)$ be the resolution of $\oM(3,6)$ with fibers
\begin{align*}
  E_{12,34,56} &\cong \bP^1 \times \bP^2 \text{ over } P_{12,34,56},\\
  L_{ij,kl,mn} &\cong \bP^1 \text{ for } P_{ij,kl,mn} \in S_1',\\
  \Pi_{ij,kl,mn} &\cong \bP^2 \text{ for } P_{ij,kl,mn} \in S_2.
\end{align*}
The resolution $\wM(3,6)$ is obtained from $\wM_{S_1,S_2}(3,6)$ by blowing up
the line $L_{12,34,56}$, and from $\wM_{S_1',S_2'}(3,6)$ by blowing up the plane
$\Pi_{12,34,56}$. Denote these blowups by $f,g$ respectively.

For a boundary divisor $D$ of $\oM(3,6)$, write $D \in A^*(\wM_{S_1,S_2}(3,6))$
for the class of its strict transform in $\wM_{S_1,S_2}(3,6)$, $D' \in
A^*(\wM_{S_1',S_2'}(3,6))$ for the class of its strict transform in
$\wM_{S_1',S_2'}(3,6)$, and $\wD \in A^*(\wM(3,6))$ for the class of its strict
transform in $\wM(3,6)$. Also write $D = f^*D$ and $D' = g^*D$.

\begin{lemma} \label{bigMdivs} In $A^*(\wM(3,6))$, we have
  \begin{align*}
    D_{ij,klmn} &= D_{ij,klmn}' + E_{12,34,56} \text{ for } ij=12,34,56,\\
    D_{ij,kl,mm} &= D_{ij,kl,mn}' - E_{12,34,56} \text{ for } \{ij,kl,mn\} = \{12,34,56\},\\
    D' &= D \text{ for all other boundary divisors.}
  \end{align*}
\end{lemma}

\begin{proof}
  Proposition \ref{SmallResolutions} implies
  \begin{align*}
    \wD_{ij,klmn}  &= D_{ij,klmn} - E_{12,34,56} = D_{ij,klmn}' \text{ for } ij=12,34,56,\\
    \wD_{ij,kl,mn} &= D_{ij,kl,mn} = D_{ij,kl,mn}' - E_{12,34,56} \text{ for } \{ij,kl,mn\} = \{12,34,56\}, \\
    \wD &= D = D' \text{ for all other boundary divisors.}
  \end{align*}
\end{proof}

\subsubsection{Ranks of Chow groups}

\begin{claim} \label{IterativeRanks}
  \begin{align*}
    \rk A^1(\wM_{S_1',S_2'}(3,6)) &= \rk A^1(\wM_{S_1,S_2}(3,6)), \\
    \rk A^2(\wM_{S_1',S_2'}(3,6)) &= \rk A^2(\wM_{S_1,S_2}(3,6)) + 1, \\
    \rk A^3(\wM_{S_1',S_2'}(3,6)) &= \rk A^3(\wM_{S_1,S_2}(3,6)).
  \end{align*}
\end{claim}

\begin{proof}
  Apply Corollary \ref{ChowGroupRanks} to the blowups $f : \wM(3,6) \to
  \wM_{S_1,S_2}(3,6)$ and $g : \wM(3,6) \to \wM_{S_1',S_2'}(3,6)$.
\end{proof}

% \subsubsection{Torsion}
%
%\begin{claim}
%  $A^k(\wM_{S_1,S_2}(3,6))$ is torsion-free $\iff A^k(\wM_{S_1',S_2'}(3,6))$ is
%  torsion-free.
%\end{claim}
%
%\begin{proof}
%  Apply Corollary \ref{TorsionFree} to the blowups $f : \wM(3,6) \to
%  \wM_{S_1,S_2}(3,6)$ and $g : \wM(3,6) \to \wM_{S_1',S_2'}(3,6)$.
%\end{proof}

\subsubsection{Generators of Chow rings}

\begin{claim}
  $A^*(\wM_{S_1,S_2}(3,6))$ is generated by the classes of the boundary divisors
  $\iff A^*(\wM_{S_1',S_2'}(3,6))$ is generated by the classes of the boundary
  divisors.
\end{claim}

\begin{proof}
  Apply Lemma \ref{ChowGenerators} to the blowups $f : \wM(3,6) \to
  \wM_{S_1,S_2}(3,6)$ and $g : \wM(3,6) \to \wM_{S_1',S_2'}(3,6)$.
\end{proof}

\subsubsection{Restriction to exceptional locus}

\begin{lemma} \label{KernelRestrictions} Let $J = \ker(A^*(\wM_{S_1,S_2}(3,6))
  \to A^*(L_{12,34,56}))$ and $J' = \ker(A^*(\wM_{S_1',S_2'}(3,6)) \to
  A^*(\Pi_{12,34,56}))$. Then $f^*J = g^*J'$.
\end{lemma}

\begin{proof}
  This follows from Proposition \ref{RestrictionMaps} together with Lemma
  \ref{bigMdivs}.
\end{proof}

\subsubsection{Presentations of Chow rings}
\begin{claim} \label{IterativeChow} $A^*(\wM_{S_1,S_2}(3,6))$ has the
  presentation of Theorem \ref{ChowResolutions} $\iff A^*(\wM_{S_1',S_2'}(3,6))$
  does.
\end{claim}

\begin{proof}
  We show only the forward direction; the backward direction is identical.

  Viewing $\wM(3,6)$ as the blowup of $\wM_{S_1,S_2}(3,6)$ along $L_{12,34,56}$,
  Keel's formula implies that
  \[
    A^*(\wM(3,6)) = \frac{A^*(\wM_{S_1,S_2}(3,6))[E_{12,34,56}]}{\text{the
        following relations}}
  \]
  \begin{enumerate}
  \item $(D_{12,3456} - E_{12,34,56})(D_{34,1256} - E_{12,34,56})(D_{56,1234} -
    E_{12,34,56}) = 0$.
  \item $J \cdot E_{12,34,56} = 0$, where $J = \ker(A^*(\wM_{S_1,S_2}(3,6)) \to
    A^*(L_{ij,kl,mn}))$.
  \end{enumerate}
  Given the assumed description of $A^*(\wM_{S_1,S_2}(3,6))$, this expands to
  \[
    A^*(\wM(3,6)) =
    \frac{\bZ[D_{ijk,lmn},D_{ij,klmn},D_{ij,kl,mn},E_{12,34,56}]}{\text{the
        following relations}}
  \]
  \begin{enumerate}
  \item (Linear relations) $f^*(0) = f^*(1) = f^*(\infty)$ for any composition
    $f: \wM(3,6) \to \oM(3,6) \to \oM_{0,4}$.
  \item (Multiplicative relations)
    \begin{enumerate}
    \item (Relations from $\wM_{S_1,S_2}(3,6)$) See Remark
      \ref{BoundaryIntersections}.
    \item (Relations from the blowup)
      \begin{enumerate}
      \item $(D_{12,3456} - E_{12,34,56})(D_{34,1256} -
        E_{12,34,56})(D_{56,1234} - E_{12,34,56}) = 0$.
      \item $J \cdot E_{12,34,56} = 0$.
      \end{enumerate}
    \end{enumerate}
  \end{enumerate}
  By Lemma \ref{bigMdivs}, this can be rewritten as
  \[
    A^*(\wM(3,6)) =
    \frac{\bZ[D_{ijk,lmn}',D_{ij,klmn}',D_{ij,kl,mn}',E_{12,34,56}]}{\text{the
        following relations}}
  \]
  \begin{enumerate}
  \item (Linear relations) $f^*(0) = f^*(1) = f^*(\infty)$ for any composition
    $f: \wM(3,6) \to \oM(3,6) \to \oM_{0,4}$.
  \item (Multiplicative relations)
    \begin{enumerate}
    \item (Relations from $\wM_{S_1,S_2}(3,6)$)
      \begin{enumerate}
      \item (Relations from $\oM(3,6)$) From the relations $J \cdot E_{12,34,56}
        = 0$ coming from the blowup, one sees that $\prod D_i = 0 \iff \prod
        D_i' = 0$ for all relations coming from $\oM(3,6)$.
      \item (Relations from $S_1$)
        \begin{enumerate}
        \item $D_{ij,kl,mn}'D_{ij,mn,kl}' = 0$ for $P_{ij,kl,mn} \in S_1'$.
        \item $(D_{12,34,56}' - E_{12,34,56})(D_{12,56,34}' - E_{12,34,56}) =
          0$.
        \end{enumerate}
      \item (Relations from $S_2$) $D_{ij,klmn}'D_{kl,ijmn}'D_{mn,ijkl}' = 0$
        for $P_{ij,kl,mn} \in S_2$.
      \end{enumerate}
    \item (Relations from the blowup)
      \begin{enumerate}
      \item $D_{12,3456}'D_{34,1256}'D_{56,1234}' = 0$,
      \item $J' \cdot E_{12,34,56} = 0$, where $J' = \ker
        (A^*(\wM_{S_1',S_2'}(3,6)) \to A^*(\Pi_{ij,kl,mn}))$. This follows by
        Lemma \ref{KernelRestrictions}.
      \end{enumerate}
    \end{enumerate}
  \end{enumerate}
  Notice that these relations consist exactly of the desired relations on
  $A^*(\wM_{S_1',S_2'}(3,6))$, plus the extra relations
  \begin{align*}
    (D_{12,34,56}' - E_{12,34,56})(D_{12,56,34}' - E_{12,34,56}) = 0 \;\; \text{ and } J \cdot E_{12,34,56} = 0.
  \end{align*}
  These extra relations are exactly the relations occuring from Keel's formula
  for the Chow ring of the blowup $g : \wM(3,6) \to \wM_{S_1',S_2'}(3,6)$ It
  follows that $A^*(\wM_{S_1',S_2'}(3,6))$ has the desired presentation.
\end{proof}

\subsubsection{Picard groups}

\begin{claim}
  $\Pic \wM_{S_1,S_2}(3,6)$ has the desired description of Theorem
  \ref{ChowResolutions} $\iff \Pic \wM_{S_1',S_2'}(3,6)$ does.
\end{claim}

\begin{proof}
  This is immediate from Claim \ref{IterativeChow} and the observation that the
  linear relations are the same on the Chow ring of any small resolution.
\end{proof}

\subsubsection{Homological results}

This section is over $\bC$.

\begin{claim} \label{IterativeHI} $\wM_{S_1,S_2}(3,6)$ is an HI scheme $\iff
  \wM_{S_1',S_2'}(3,6)$ is an HI scheme.
\end{claim}

\begin{proof}
  Apply Lemma \ref{HIschemes} to the blowups $f : \wM'(3,6) \to
  \wM_{S_1,S_2}(3,6)$ and $g : \wM'(3,6) \to \wM_{S_1',S_2'}(3,6)$, noting that
  any line or plane is an HI scheme so the centers of these blowups are HI
  schemes.
\end{proof}

\subsubsection{Outline of proof of Theorem \ref{ChowResolutions}}

The above claims prove Proposition \ref{IterativeProposition}. By this
proposition it is now enough to prove Theorem \ref{ChowResolutions} holds for
the small resolution $\wM_1(3,6)$. We will show this in the next two sections.
In Section \ref{PresentationM1}, we will show the first part of Theorem
\ref{ChowResolutions}, i.e. the presentation of $A^*(\wM_1(3,6))$. This is the
most involved part of the proof. In Section \ref{RemainingM1}, we will establish
the remaining parts of Theorem \ref{ChowResolutions}.

\subsection{Presentation of $A^*(\wM_1(3,6))$} \label{PresentationM1}

\subsubsection{Obvious relations}

We call the desired relations on $A^*(\wM_1(3,6))$ from Theorem
\ref{ChowResolutions} the ``obvious relations'' on $A^*(\wM_1(3,6))$. This is
because each such relation obviously has to hold: it is either the pullback of a
relation on $\oM_{0,4}$, or reflects the fact that two boundary divisors are
disjoint.

\subsubsection{Outline of proof}

The first part of Theorem \ref{ChowResolutions} asserts that $A^*(\wM_1(3,6))$
is generated by boundary divisors, and the only relations are the obvious ones.
The Chow ring of $\wM_1(3,6)$ can be determined by repeated applications of
Keel's formula to the sequence of blowups of Theorem \ref{BlowupConstructions}.
The presentation obtained in this way does not immediately look like the desired
presentation. Instead, we show that each relation obtained from this procedure
is contained in the ideal generated by the obvious relations.

Computing the relations obtained by Keel's formula can be quite involved; we
separate the computations by the Chern polynomial relations and restriction
relations.

Recall from Section \ref{forgetandrestrict} that any boundary divisor on
$\oM_{0,5}$ can be written as $D_{ij} = D_{ij,klm}$. Let $p_n : X_1 \to
\oM_{0,5}$, $n=1,2$ denote the two natural projections.

\subsubsection{Chow ring of $A^*(X_1)$}

By \cite[Theorem 2.2]{keelIntersectionTheoryModuli1992}, $A^*(X_1) \cong
A^*(\oM_{0,5}) \otimes A^*(\oM_{0,5})$. Recall from Section \ref{blowupNotation}
that
\begin{align*}
  D_{ij5,kl6}^1 = p_1^*(D_{kl}), \;\; D_{ij6,kl5}^1 = p_2^*(D_{kl}), \\
  D_{i5,jkl6}^1 = p_1^*(D_{i5}), \;\; D_{i6,jkl5}^1 = p_2^*(D_{i6}).
\end{align*}
The additional divisors $D_{i56,jkl}^1$ on $X_1$ are the strict transforms of
the divisors $Q_{jkl} \subset X_0$ under the blowup $X_1 \to X_0$, from which we
compute the linear equivalence.
\[
  D_{i56,jkl}^1 = D_{ij5,kl6}^1 + D_{j5,ikl6}^1 + D_{ij6,kl5}^1 + D_{j6,ikl5}^1.
\]
Interpreting Keel's presentation of $A^*(\oM_{0,5})$ in this notation, we
conclude that
\begin{align*}
  A^*(X_1) &= \frac{\bZ[D_{i6,jkl5}^1,D_{i5,jkl6}^1,D_{ij6,kl5}^1,D_{ij5,kl6}^1,D_{i56,jkl}^1]}{\text{the following relations}}.
\end{align*}
\begin{enumerate}
\item (Linear relations)
  \begin{enumerate}
  \item $p_n^*(D_{ij} + D_{kl}) = p_n^*(D_{ik} + D_{jl}) = p_n^*(D_{il} +
    D_{jk})$ for $n=1,2$, $i,j,k,l$ distinct.
  \item $D_{i56,jkl}^1 = D_{ij5,kl6}^1 + D_{j5,ikl6}^1 + D_{ij6,kl5}^1 +
    D_{j6,ikl5}^1$.
  \end{enumerate}
\item (Multiplicative relations) $p_n^*(D_{ij})p_n^*(D_{ab}) = 0$ for $n=1,2$
  and $\lvert ij \cap ab \rvert = 1$.
\end{enumerate}

\subsubsection{Relations on $A^*(\wM_1(3,6))$ coming from $A^*(X_1)$}

Applying Lemma \ref{ChowGenerators} and Keel's formula to our sequence of
blowups, we see that $A^*(\wM_1(3,6))$ is generated over $A^*(X_1)$ by
\[
  D_{ijk,lmn},D_{ij,klmn},D_{ij,kl,mn},
\]
with the linear relations
\begin{align*}
  D_{ij5,kl6}^1 = r_6^*(D_{kl}), \;\; D_{ij6,kl5}^1 = r_5^*(D_{kl}), \\
  D_{i5,jkl6}^1 = r_6^*(D_{i5}), \;\; D_{i6,jkl5}^1 = r_5^*(D_{i6}), \\
  D_{i56,jkl}^1 = f_i^*(D_{56} + D_{jk} + D_{jl} + D_{kl}).
\end{align*}
(Recall the pullback formulas from Proposition \ref{PullbackFormulas}.) It
follows that $A^*(\wM_1(3,6))$ is generated over $\bZ$ by
\[
  D_{ijk,lmn},D_{ij,klmn},D_{ij,kl,mn},
\]
and the relations on $A^*(X_1)$ from the previous section give the following
relations on $A^*(\wM_1(3,6))$.
\begin{enumerate}
\item (Linear relations)
  \begin{enumerate}
  \item $r_n^*(D_{ij} + D_{kl}) = r_n^*(D_{ik} + D_{jl}) = r_n^*(D_{il} +
    D_{jk})$ for $n=5,6$,
  \item $f_i^*(D_{56} + D_{jk} + D_{jl} + D_{kl}) = r_6^*(D_{kl} + D_{j5}) +
    r_5^*(D_{kl} + D_{j6})$ for $i=1,2,3,4$.
  \end{enumerate}
  We leave it to the reader to verify that the second class of relations are
  contained in the obvious linear relations on $A^*(\wM_1(3,6))$.
\item Multiplicative relations: $r_n^*(D_{ij})r_n^*(D_{ab}) = 0$ for $n=5,6$ and
  $\lvert ij \cap ab \rvert = 1$. We again leave it to the reader to verify that
  all these relations are contained in the obvious relations on
  $A^*(\wM_1(3,6))$.
\end{enumerate}

\subsubsection{Chern polynomial relations}
By Lemmas \ref{CompleteIntersections} and \ref{Chernformulas}, we obtain the
following relations from the Chern polynomials of the blown up centers $f_k :
X_k \to X_{k-1}$, $k \neq 4$.
\begin{enumerate}
\item $(D_{l5,ijk6}^1 - D_{ijk,l56}^2)(D_{l6,ijk5}^1 - D_{ijk,l56}^2) = 0$. This
  becomes
  \[
    (r_6^*(D_{l5}) - D_{ijk,l56})(r_5^*(D_{l6}) - D_{ijk,l56}) = 0.
  \]
\item $(D_{kl5,ij6}^2 - D_{ij,kl56}^3)(D_{kl6,ij5}^2 - D_{ij,kl56}^3) = 0$.
  These become
  \[
    (r_6^*(D_{ij}) - D_{ij,kl56} - D_{ij,kl,56})(r_5^*(D_{ij}) - D_{ij,kl56} -
    D_{ij,kl,56}) = 0.
  \]
\item $(D_{klm,ijn}^4 - D_{ij,kl,mn}^5)(D_{mn,ijkl}^4 - D_{ij,kl,mn}^5) = 0$.
  These become
  \[
    D_{klm,ijn}D_{mn,ijkl} = 0.
  \]
\end{enumerate}
In any case, these relations are all contained in the obvious relations on
$A^*(\wM_1(3,6))$.

The only blowup whose Chern polynomial relation we have not described is the
blowup $f_4: X_4 \to X_3$ along $D_{56,1234}^3$. Unlike the other cases,
$D_{56,1234}^3$ is not a complete intersection, so its Chern polynomial is more
difficult to describe.

\begin{lemma} \label{Chern4} The Chern polynomial relation coming from the
  blowup of $D_{56,1234}^3$ in $X_3$ is contained in the obvious relations on
  $A^*(\wM_1(3,6))$.
\end{lemma}

\begin{proof}
  Notice that $D_{56,1234}^0 = \Delta(\bP^2) \subset \bP^2 \times \bP^2$. The
  Chern class of its normal bundle is therefore $(1 + h)^3$, where $h$ is the
  generator of $A^*(\bP^2)$. Since $D_{56,1234}^0$ intersects each center of
  $X_1 \to X_0$ transversally in a point, we find that $D_{56,1234}^1 =
  \Delta(\oM_{0,5}) \subset \oM_{0,5} \times \oM_{0,5}$, and the Chern class of
  its normal bundle is still $(1 + h)^3$. View $D_{56,1234}^1$ as $Bl_4(\bP^2)
  \cong \oM_{0,5}$, and write $e_i$ for the classes of the exceptional divisors
  in $A^*(D_{56,1234}^1)$.

  We split the blowups $X_3 \to X_2 \to X_1$ into
  \[
    X_3 \to X_2' \to X_2 \to X_1,
  \]
  where $X_2' \to X_2$ is the blowup along the $D_{ij,kl56}^2$ with
  $ij=12,13,23$, and $X_3 \to X_2'$ is the blowup along the strict transforms of
  the $D_{ij,kl56}^{2}$ for $ij=14,24,34$. This makes each blowup along a
  disjoint union of smooth irreducible centers. A given center $D_{ijk,l56}^1$
  intersects $D_{56,1234}^1$ in a Cartier divisor of class $e_l$ on
  $D_{56,1234}^1$; the intersection is transversal inside of $D_{i56,jkl}^1$.
  Likewise, a given center $D_{ij,kl56}^2$ intersects $D_{56,1234}^2$ in a
  Cartier divisor of class $h - e_k - e_l$ on $D_{56,1234}^2$, and the
  intersection is transversal inside of $D_{k56,ijl}^2$. It follows that
  $D_{56,1234}^3 \cong \oM_{0,5}$, and repeated applications of Lemma
  \ref{Chernbad} show that the Chern class of its normal bundle is
  \[
    1 - 3h + \sum e_i - 3h^2 + \sum e_i^2.
  \]

  It remains to determine the class of $D_{56,1234}^3$ in $X_3$.
  
  Since $D_{56,1234}^0$ is the diagonal in $\bP^2 \times \bP^2$, it has class
  \[
    (H_1)^2 + H_1H_2 + (H_2)^2,
  \]
  where $H_i$ are the generators of $\bP^2 \times \bP^2$. Since $D_{56,1234}^0$
  intersects the centers of the blowup $X_1 \to X_0$ transversally in a point
  each, it follows by \cite[Corollary 6.7.2]{fultonIntersectionTheory1998} that
  this is also the class of $D_{56,1234}^1$ in $X_1$. In our presentation of
  $A^*(X_1)$ we can write
  \[
    H_1 = p_1^*(D_{12} + D_{35} + D_{45}) \;\; \text{ and } \;\; H_2 =
    p_2^*(D_{12} + D_{36} + D_{46}).
  \]
  Now let $D_2 = \sum D_{ijk,l56}^2$ be the exceptional divisor of $X_2 \to
  X_1$, let $D_3^1 = \sum_{ij=12,13,23} D_{ij,kl56}^{2'}$ be the exceptional
  divisor of $X_2' \to X_2$, and let $D_3^2 = \sum_{ij=14,24,34} D_{ij,kl56}^3$
  be the exceptional divisor of $X_3 \to X_2$. Let $D_3 = D_3^1 + D_3^2$.
  Repeated applications of \cite[Theorem 6.7]{fultonIntersectionTheory1998} show
  that
  \[
    D_{56,1234}^3 = (H_1)^2 + H_1H_2 + (H_2)^2 + (D_2)^2 + D_3^1D_3^2 - H_1(D_2
    + D_3) - H_2(D_2 + D_3) + D_2D_3.
  \]

  Examining the restriction map $A^*(X_3) \to A^*(D_{56,1234}^3)$, we see that a
  Chern polynomial for $D_{56,1234}^3$ in $X_3$ is
  \[
    t^2 + (2H_1 + H_2 - 2D_2 - D_3)t + D_{56,1234}^3,
  \]
  so that the Chern polynomial relation in $A^*(\wM_1(3,6))$ becomes
  \[
    (D_{56,1234})^2 - (2H_1 + H_2 - 2D_2 - D_3)D_{56,1234} + D_{56,1234}^3 = 0,
  \]
  where
  \[
    H_1 = r_6^*(D_{12} + D_{35} + D_{45}) \;\; \text{ and } \;\; H_2 =
    r_5^*(D_{12} + D_{36} + D_{46}),
  \]
  and
  \begin{align*}
    D_2 = \sum D_{ijk,l56}, \;\; D_3^1 = \sum_{ij=12,13,23} (D_{ij,kl56} + D_{ij,kl,56}), \;\; D_3^2 = \sum_{ij=14,24,34} (D_{ij,kl56} + D_{ij,kl,56}),
  \end{align*}
  and $D_3$, $D_{56,1234}^3$ are as above.

  It is a direct verification that this relation is already zero on
  $A^*(\wM_1(3,6))$.
\end{proof}

\subsubsection{Restriction relations}

For a stratum $D^k$ of $X_k$, we write $J_D^k$ for the kernel of the restriction
map $A^*(X_k) \to A^*(D^k)$.

\begin{lemma}
  Let $D^1$ be any 2-stratum in $X_1$. Then $A^*(X_1) \to A^*(D^1)$ is
  surjective, and the relations on $A^*(\wM_1(3,6))$ coming from $J_D^1$ are
  contained in the obvious relations.
\end{lemma}

\begin{proof}
  \begin{enumerate}
  \item If $D^1 = D_{56,1234}^1$, then $D^1$ is the diagonal in $X_1 = \oM_{0,5}
    \times \oM_{0,5}$. It follows that $A^*(X_1) \to A^*(D^1)$ is surjective,
    and
    \[
      J_D^1 = \langle p_1^*(D_{ij}) - p_2^*(D_{ij}) \rangle.
    \]
    The relations on $A^*(\wM_1(3,6))$ coming from $J_D^1$ therefore look like
    \[
      r_6^*(D_{ij})D_{56,1234} = r_5^*(D_{ij})D_{56,1234}.
    \]
    These follow from the obvious relations.
  \item If $D^1 \neq D_{56,1234}^1$, then $D^1 = D_{a_1b_1} \times D_{a_2b_2}
    \subset \oM_{0,5} \times \oM_{0,5}$ for some $a_i,b_i$ (not necessarily
    distinct). Keel's results imply that the restriction map $A^*(X_1) \to
    A^*(D^1)$ is the tensor product of the restriction maps $A^*(\oM_{0,5}) \to
    D_{a_ib_i}$, $i=1,2$, and $J_D^1$ is generated by
    \[
      r_6^*(D_{mn}) \text{ for } \lvert mn \cap a_1b_1 \rvert = 1 \;\; \text{
        and }\;\; r_5^*(D_{mn}) \text{ for } \lvert mn \cap a_2b_2 \rvert = 1.
    \]

    If $D^1 = D_{ijk,l56}^1$, then the relations on $A^*(\wM_1(3,6))$ coming
    from $J_D^1$ look like
    \[
      r_6^*(D_{mn})D_{ijk,l56} = 0 \text{ for } \lvert mn \cap l5 \rvert = 1
      \text{ and } r_5^*(D_{mn})D_{ijk,l56} = 0 \text{ for } \lvert mn \cap l6
      \rvert = 1.
    \]

    If $D^1 = D_{ij,kl56}^1$, then the relations on $A^*(\wM_1(3,6))$ coming
    from $J_D^1$ look like
    \[
      r_6^*(D_{mn})(D_{ij,kl56} + D_{ij,kl,56}) = r_5^*(D_{mn})(D_{ij,kl56} +
      D_{ij,kl,56}) = 0 \text{ for } \lvert mn \cap ij \rvert = 1.
    \]

    If $D^1 = D_{ij,k5,l6}^1$, then the relations on $A^*(\wM_1(3,6))$ coming
    from $J_D^1$ look like
    \[
      r_6^*(D_{mn})D_{ij,k5,l6} = 0 \text{ for } \lvert mn \cap ij \rvert = 1
      \text{ and } r_5^*(D_{mn})D_{ij,k5,l6} = 0 \text{ for } \lvert mn \cap l6
      \rvert = 1.
    \]

    If $D^1 = D_{ij,k6,l5}^1$, then the relations on $A^*(\wM_1(3,6))$ coming
    from $J_D^1$ look like
    \[
      r_6^*(D_{mn})D_{ij,k6,l5} = 0 \text{ for } \lvert mn \cap l5 \rvert = 1
      \text{ and } r_5^*(D_{mn})D_{ij,k6,l5} = 0 \text{ for } \lvert mn \cap ij
      \rvert = 1.
    \]

    In any case, all relations follow from the obvious relations on
    $A^*(\wM_1(3,6))$.
  \end{enumerate}
\end{proof}

Recall $f_2 : X_2 \to X_1$ is the blowup of $X_1$ along $D_{ijk,l56}^1$. In
addition to the relations from $J_{D_{ijk,l56}}^1$, this blowup also gives
relations
\[
  D_{ijk,l56}D_{abc,d56} = 0 \text{ for } l \neq d
\]
on $A^*(\wM_1(3,6))$ (because any two $D_{ijk,l56}^1$ are disjoint; these
relations appear from kernels of restriction maps by viewing $f_2 : X_2 \to X_1$
as an iterated blowup). These relations are also contained in the obvious
relations.

\begin{lemma}
  Let $D^2$ be any 2-stratum in $X_2$. Then $A^*(X_2) \to A^*(D^2)$ is
  surjective, and the relations on $A^*(\wM_1(3,6))$ coming from $J_D^2$ are
  contained in the obvious relations.
\end{lemma}

\begin{proof}
  \begin{enumerate}
  \item If $D^2 = D_{ij,kl56}^2$, then $D^1$ intersects $D_{ijk,l56}^1$ and
    $D_{ijl,k56}^1$ transversally in a point each, and is disjoint from the
    other centers of $f_2 : X_2 \to X_1$. By Lemma \ref{restrictionformulas},
    $A^*(X_2) \to A^*(D^2)$ is surjective, with kernel generated by the
    following.
    \begin{enumerate}
    \item $J_D^1$. The induced relations $J_D^1(D_{ij,kl56} + D_{ij,kl,56}) = 0$
      in $A^*(\wM_1(3,6))$ are contained in the obvious relations by the
      previous lemma.
    \item $D_{ikl,j56}^2,D_{jkl,i56}^2$. The induced relations
      \[
        D_{ikl,j56}(D_{ij,kl56} + D_{ij,kl,56}) = D_{jkl,i56}(D_{ij,kl56} +
        D_{ij,kl,56}) = 0
      \]
      in $A^*(\wM_1(3,6))$ are contained in the obvious relations.
    \item $\alpha \cdot D_{ijk,l56}^2, \alpha \cdot D_{ijl,k56}^2$ for $\alpha
      \in A^*(X_1)$ with $\deg \alpha > 0$. The induced relations on
      $A^*(\wM_1(3,6))$ look like
      \[
        r_m^*(D_{I,J})D_{ijk,l56}(D_{ij,kl56} + D_{ij,kl,56}) = 0 \text{ for }
        m=5,6.
      \]
      These are also contained in the obvious relations on $A^*(\wM_1(3,6))$,
      but not as obviously as the other cases.
      \begin{example} \label{badrels} Using the linear relations, write
        \[
          r_6^*(D_{45})D_{123,456}(D_{12,3456} + D_{12,34,56}) = r_6^*(D_{14} +
          D_{35} - D_{13})D_{123,456}(D_{12,3456} + D_{12,34,56}).
        \]
        Expanding the right-hand side using the pullback formulas (Proposition
        \ref{PullbackFormulas}), one sees that this is already zero by the
        obvious relations.
      \end{example}
      We leave the similar calculations for the remaining cases to the reader.
    \end{enumerate}
  \item If $D^2 = D_{56,1234}^2$, then $D^1$ intersects each $D_{ijk,l56}^1$
    nontransversally in a line, which is a Cartier divisor on $D^1$. The class
    $D_{l5,ijk6}^1 \in A^1(X_1)$ restricts to the class of this Cartier divisor
    on $A^*(D^1)$. By Lemma \ref{restrictionformulas}, $A^*(X_2) \to A^*(D^2)$
    is surjective with kernel generated by the following.
    \begin{enumerate}
    \item $J_D^1$. The induced relations $J_D^1D_{56,1234} = 0$ on
      $A^*(\wM_1(3,6))$ are contained in the obvious relations by the previous
      lemma.
    \item $D_{ijk,l56}^2 - D_{l5,ijk6}^1$. The induced relations on
      $A^*(\wM_1(3,6))$ look like
      \[
        (D_{ijk,l56} - r_6^*(D_{l5}))D_{56,1234} = 0,
      \]
      which are contained in the obvious relations.
    \end{enumerate}
  \item If $D^2 = D_{ij,k5,l6}^2$, then $D^1$ intersects $D_{ijk,l56}^1$
    nontransversally in a line, which is a Cartier divisor on $D^1$, and is
    disjoint from the other centers of $f_2 : X_2 \to X_1$. The class
    $D_{l5,ijk6}^1$ in $A^1(X_1)$ restricts to the class of $D^1 \cap
    D_{ijk,l56}^1$ in $A^*(D^1)$. By Lemma \ref{restrictionformulas}, $A^*(X_2)
    \to A^*(D^2)$ is surjective with kernel generated by the following.
    \begin{enumerate}
    \item $J_D^1$. The induced relations $J_D^1D_{ij,k5,l6} = 0$ on
      $A^*(\wM_1(3,6))$ are contained in the obvious relations by the previous
      lemma.
    \item $D_{abc,d56}^2$ for $d \neq l$. The induced relations
      $D_{abc,d56}D_{ij,k5,l6} = 0$ on $A^*(\wM_1(3,6))$ are contained in the
      obvious relations.
    \item $D_{ijk,l56}^2 - D_{l5,ijk6}^1$. The induced relations on
      $A^*(\wM_1(3,6))$ look like
      \[
        (D_{ijk,l56} - r_6^*(D_{l5}))D_{ij,k5,l6} = 0,
      \]
      which are contained in the obvious relations.
    \end{enumerate}
  \item The case $D^2 = D_{ij,k6,l5}^2$ is symmetric to the case $D^2 =
    D_{ij,k5,l6}^2$.
  \end{enumerate}
\end{proof}

Recall that $f_3 : X_3 \to X_2$ is the blowup along the $D_{ij,kl56}^2$. There
is a subtlety here which does not occur in the other blowups, which is that the
$D_{ij,kl56}^2$ intersect. Since they all intersect transversally, we can view
$f_3 : X_3 \to X_2$ as the iterated blowup of the $D_{ij,kl56}^2$ in any order,
but by Lemma \ref{restrictionformulas} we get additional relations on
$A^*(\wM_1(3,6))$ of the following forms.
\begin{enumerate}
\item $(D_{ij,kl56} + D_{ij,kl,56})(D_{ab,cd56} + D_{ab,cd,56}) = 0$ for $\lvert
  ij \cap ab \rvert = 1$. These are contained in the obvious relations.
\item $\alpha (D_{ij,kl56} + D_{ij,kl,56})(D_{kl,ij56} + D_{kl,ij,56}) = 0$ for
  $\alpha \in A^*(X_2)$ with $\deg \alpha > 0$. These are also contained in the
  obvious relations by a calculation similar to Example \ref{badrels} above.
\end{enumerate}

\begin{lemma}
  Let $D^3$ be any 2-stratum in $X_3$. Then $A^*(X_3) \to A^*(D^3)$ is
  surjective, and the relations on $A^*(\wM_1(3,6))$ coming from $J_D^3$ are
  contained in the obvious relations.
\end{lemma}

\begin{proof}
  \begin{enumerate}
  \item If $D^3 = D_{56,1234}^3$, then $D^2$ intersects each $D_{ij,kl56}^2$
    nontransversally in a line, which is a Cartier divisor on $D^2$, and the
    class $D_{kl5,ij6}^2 \in A^1(X_2)$ restricts to this Cartier divisor on
    $D^2$. By Lemma \ref{restrictionformulas}, $A^*(X_3) \to A^*(D^3)$ is
    surjective with kernel generated by the following.
    \begin{enumerate}
    \item $J_D^2$. The induced relations $J_D^2D_{56,1234} = 0$ in
      $A^*(\wM_1(3,6))$ are contained in the obvious relations by the previous
      lemma.
    \item $D_{ij,kl56}^3 - D_{kl5,ij6}^2$. The induced relations on
      $A^*(\wM_1(3,6))$ look like
      \[
        (D_{ij,kl56} + D_{ij,kl,56} - r_6^*(D_{ij,kl5}))D_{56,1234} = 0,
      \]
      which are contained in the obvious relations.
    \end{enumerate}
  \item If $D^3 = D_{ij,k5,l6}^3$, then $D^2$ intersects $D_{ij,kl56}^2$
    nontransversally in a line, which is a Cartier divisor on $D^2$, and is
    disjoint from the remaining centers. The class $D_{kl6,ij5}^2 \in A^1(X_2)$
    restricts to the class of $D^2 \cap D_{ij,kl56}^2$ in $A^*(D^2)$. By Lemma
    \ref{restrictionformulas}, $A^*(X_3) \to A^*(D^3)$ is surjective with kernel
    generated by the following.
    \begin{enumerate}
    \item $J_D^2$. The induced relations $J_D^2D_{ij,k5,l6} = 0$ in
      $A^*(\wM_1(3,6))$ are contained in the obvious relations by the previous
      lemma.
    \item $D_{ab,cd56}^3$ for $ab \neq ij$. The induced relations
      \[
        (D_{ab,cd56} + D_{ab,cd,56})D_{ij,k5,l6} = 0
      \]
      on $A^*(\wM_1(3,6))$ are contained in the obvious relations.
    \item $D_{ij,kl56}^3 - D_{kl6,ij5}^2$. The induced relation on
      $A^*(\wM_1(3,6))$ looks like
      \[
        (D_{ij,kl56} + D_{ij,kl,56} - r_5^*(D_{ij,kl6}))D_{ij,k5,l6} = 0,
      \]
      which is contained in the obvious relations.
    \end{enumerate}
  \item The case $D^3 = D_{ij,k6,l5}^3$ is symmetric to the case
    $D_{ij,k5,l6}^3$.
  \item Suppose $D^3 = D_{ij,kl,56}^3$. Note that $D^2$ is a line contained in
    $D_{ij,kl56}^2$ in $X_2$. Factor $f_3 : X_3 \to X_2$ into $X_3 \to X_2' \to
    X_2$, where $X_2' \to X_2$ is the blowup along all centers besides
    $D_{ij,kl56}^2$, and $X_3 \to X_2'$ is the blowup along the strict transform
    $D_{ij,kl56}^{2'}$ of $D_{ij,kl56}^{2'}$, and let $D^{2'}$ be the strict
    transform of $D^2$ in $X_2'$. Then $D_{ij,kl56}^{2'} \cong \oM_{0,5}$ and
    $D^{2'} \cong D_{ij,klm} \subset \oM_{0,5}$. In particular
    $A^*(D_{ij,kl56}^{2'}) \to A^*(D^{2'})$ is surjective, and we can compute
    its kernel. By Lemma \ref{restrictionformulas}, we determine that $A^*(X_3)
    \to A^*(D^3)$ is surjective, with kernel generated by the following.
    \begin{enumerate}
    \item $J_{D_{ij,kl56}}^3$. This is described by the previous calculations,
      and the induced relations $J_{D_{ij,kl56}}^3D_{ij,kl,56} = 0$ on
      $A^*(\wM_1(3,6))$ are contained in the obvious relations.
    \item
      $D_{k5,ijl6}^3,D_{k6,ijl5}^3,D_{l5,ijk6}^3,D_{l6,ijk5}^3,D_{ij5,kl6}^3,D_{ij6,kl5}^3$.
      The induced relations on $A^*(\wM_1(3,6))$ are all contained in the
      obvious relations.
    \end{enumerate}
  \end{enumerate}
\end{proof}

Recall $f_4 : X_4 \to X_3$ is the blowup of $X_3$ along $D_{56,1234}^3$.
\begin{lemma}
  Let $D^4$ be any 2-stratum in $X_4$. Then $A^*(X_4) \to A^*(D^4)$ is
  surjective, and the relations on $A^*(\wM_1(3,6))$ coming from $J_D^4$ are
  contained in the obvious relations.
\end{lemma}

\begin{proof}
  \begin{enumerate}
  \item If $D^4 = D_{ij,k5,l6}^4$, then $D^3$ is disjoint from $D_{56,1234}^3$,
    so $A^*(X_4) \to A^*(D_4)$ is surjective with kernel generated by the
    following.
    \begin{enumerate}
    \item $J_D^3$. The induced relations $J_D^3D_{ij,k5,l6} = 0$ on
      $A^*(\wM_1(3,6))$ are contained in the obvious relations by the previous
      lemma.
    \item $D_{56,1234}^4$. The induced relation $D_{56,1234}D_{ij,k5,l6} = 0$ on
      $A^*(\wM_1(3,6))$ is contained in the obvious relations.
    \end{enumerate}
  \item The case $D^4 = D_{ij,k6,l5}^4$ is symmetric to the case $D^4 =
    D_{ij,k5,l6}^4$.
  \item If $D^4 = D_{ij,kl,56}^4$, then $D^3$ intersects $D_{56,1234}^3$
    nontransversally in a line, which is a Cartier divisor on $D^4$. The class
    $D_{i56,jkl}^3 \in A^1(X_3)$ restricts to the class of this Cartier divisor
    in $A^*(D^3)$. Thus $A^*(X_4) \to A^*(D^4)$ is surjective, with kernel
    generated by the following.
    \begin{enumerate}
    \item $J_D^3$. By the previous lemma, the induced relations
      $J_D^3D_{ij,kl,56} = 0$ on $A^*(\wM_1(3,6))$ are contained in the obvious
      relations.
    \item $D_{56,1234}^4 - D_{i56,jkl}^3$. The induced relation
      \[
        (D_{56,1234}^4 - D_{i56,jkl}^3)D_{ij,kl,56} = D_{i56,jkl}^4D_{ij,kl,56}
        = 0,
      \]
      is contained in the obvious relations.
    \end{enumerate}
  \end{enumerate}
\end{proof}

Recall $f_5 : X_5 \to X_4$ is the blowup of $X_4$ along all $D_{ij,kl,mn}^4$.
These are all disjoint, so in addition to the relations from the previous
lemmas, we get relations
\[
  D_{ij,kl,mn}D_{ab,cd,ef} = 0 \text{ for } D_{ij,kl,mn} \neq D_{ab,cd,ef}.
\]
These are all contained in the obvious relations. Thus we have shown that all
restriction relations on $A^*(\wM_1(3,6))$ are contained in the obvious
relations.

\subsubsection{Chow ring of $\wM_1(3,6)$}
The above subsections show that the Chow ring $A^*(\wM_1(3,6))$ of the small
resolution $\wM_1(3,6)$ has the desired presentation of Theorem
\ref{ChowResolutions}.

\subsection{Remaining results} \label{RemainingM1}

Recall from Section \ref{BlowupCenters} the description of the centers of the
blowup sequence $\wM_1(3,6) \to \bP^2 \times \bP^2$.
\begin{observation}
  The irreducible centers of the blowup sequence $\wM_1(3,6) \to \bP^2 \times
  \bP^2$ all look like either $Bl_k\bP^2$ for some $0 \leq k \leq 4$, or
  $Bl_k(\bP^1 \times \bP^1)$ for some $0 \leq k \leq 3$.
\end{observation}

\subsubsection{Ranks of Chow groups} \label{ChowRankM1}

From the observation we can easily compute the ranks of the Chow groups of all
the centers. Then from Corollary \ref{ChowGroupRanks} applied to the sequence of
blowups $\wM_1(3,6) \to \bP^2 \times \bP^2$, we compute
\begin{align*}
  \rk A^1(\wM_1(3,6)) &= 51,\\
  \rk A^2(\wM_1(3,6)) &= 127,\\
  \rk A^3(\wM_1(3,6)) &= 51.
\end{align*}

\subsubsection{Picard groups}

It follows by the description of $A^*(\wM_1(3,6))$ that $\Pic \wM_1(3,6) =
A^1(\wM_1(3,6))$ is generated by the classes of the boundary divisors, modulo
the linear relations. By the linear relations any boundary divisor class in
$A^1(\wM_1(3,6))$ can be written as a linear combination of the desired basis
elements. Since there are 51 such elements and $\rk A^1(\wM_1(3,6)) = 51$, it
follows that these elements indeed form a basis of $\Pic \wM_1(3,6)$.

\subsubsection{Homological results}

By \cite{keelIntersectionTheoryModuli1992}, $X_1 = \oM_{0,5} \times \oM_{0,5}$
is an HI scheme. Each irreducible center of the sequence of blowups $\wM_1 \to
X_1$ is also an HI scheme (cf. Section \ref{BlowupCenters}), thus $\wM_1(3,6)$
is an HI scheme by Lemma \ref{HIschemes}.

\subsubsection{Completion of proof}

The above results establish Theorem \ref{ChowResolutions} for the small
resolution $\wM_1(3,6)$. The theorem follows for all small resolutions by
Proposition \ref{IterativeProposition}.

\begin{remark}
  We independently verified the calculation of $A^*(\wM_1(3,6))$ above on a
  computer by the following method. By Lemma \ref{ChowGenerators},
  $A^*(\wM_1(3,6))$ is generated by the classes of the boundary divisors. Let
  $R$ be the ring generated by the boundary divisors, modulo the obvious
  relations. Then there is a natural surjective ring morphism $R \to
  A^*(\wM_1(3,6))$. Using a computer we verify that $\rk R^k = \rk
  A^k(\wM_1(3,6))$ for all $k$, and all $R^k$ are torsion-free. Since each
  $A^k(\wM_1(3,6))$ is also torsion-free, it follows that $R^k \to
  A^k(\wM_1(3,6))$ is an isomorphism for all $k$, hence $R \to A^k(\wM_1(3,6))$
  is torsion-free.
\end{remark}

\section{Intersection theory of $\oM(3,6)$} \label{IntersectionTheoryM}

By the Chow ring of a singular variety, we mean the operational Chow ring of
Fulton-MacPherson's bivariant intersection theory \cite[Chapter
17]{fultonIntersectionTheory1998}.

Define
\begin{align*}
  \delta_{ijk,lmn} &= D_{ijk,lmn},\\
  \delta_{ij,k,lmn} &= D_{ij,klmn} + D_{kl,ij,mn} + D_{km,ij,ln} + D_{kn,ij,lm} \text{ for } k < l,m,n,\\
  \delta_{ij,kl,mn} &= D_{ij,kl,mn} - D_{kl,ij,mn} \text{ for } k < l,m,n.
\end{align*}
The conditions $k < l,m,n$ are so we do not have to worry about permuting the
indices. Note that there are 20 $\delta_{ijk,lmn}$, 15 $\delta_{ij,k,lmn}$, and
15 $\delta_{ij,kl,mn}$. These divisors are all Cartier by Proposition
\ref{BoundaryDivs}. Also observe $r_k^*(D_{ij}) = \delta_{ijk,lmn} +
\delta_{ij,k,lmn}$.

\begin{theorem} \label{ChowM}\
  \begin{enumerate}
  \item $A^*(\oM(3,6))$ is the subring of $A^*(\wM_1(3,6))$ described by
    $A^k(\oM(3,6)) = A^k(\wM_1(3,6))$ for $k \neq 1$, and
    \[
      A^1(\oM(3,6)) = \{\alpha \in A^1(\wM_1(3,6)) \mid
      \alpha\vert_{L_{ij,kl,mn}} = 0 \text{ for all } P_{ij,kl,mn}\}.
    \] \label{ChowMdescription}
  \item The nontrivial (i.e. $\neq 0,1$) ranks of the Chow groups are
    \begin{align*}
      \rk A^1(\oM(3,6)) &= 36, \\
      \rk A^2(\oM(3,6)) &= 127, \\
      \rk A^3(\oM(3,6)) &= 51.
    \end{align*}
  \item
    \begin{enumerate}
    \item $\Pic \oM(3,6) = A^1(\oM(3,6))$ and is generated by the
      $\delta_{ijk,lmn},\delta_{ij,k,lmn},\delta_{ij,kl,mn}$, modulo the linear
      relations $f^*(0) = f^*(1) = f^*(\infty)$ for any composition $f:\oM(3,6)
      \xrightarrow{r_i} \oM_{0,5} \xrightarrow{f_j} \oM_{0,4}$.
    \item A basis for $\Pic \oM(3,6)$ is
      \begin{enumerate}
      \item
        $\delta_{156,234},\delta_{256,134},\delta_{345,126},\delta_{346,125},\delta_{356,124},\delta_{456,123}$,
      \item all 15 $\delta_{ij,k,lmn}$,
      \item all 15 $\delta_{ij,kl,mn}$.
      \end{enumerate}
    \end{enumerate}
  \item $A^*(\oM(3,6))$ is generated by $A^1(\oM(3,6))$.
  \end{enumerate}
\end{theorem}

\subsection{General procedure for determining the Chow ring of a singular
  variety}

A general procedure for determining the Chow ring of a singular variety was
given by Kimura \cite[Remark 3.2]{shun-ichiFractionalIntersectionBivariant1992}.
The key result is the following lemma.

\begin{lemma}[{\cite[Theorem
    3.1]{shun-ichiFractionalIntersectionBivariant1992}}] \label{ChowEnvelope}
  Let $\pi : \wX \to X$ be a proper birational morphism such that every closed
  subvariety of $X$ is the birational image of a closed subvariety of $\wX$.
  Suppose $\pi$ is an isomorphism outside of a closed subscheme $Z \subset X$;
  let $Z_i$ be the irreducible components of $Z$ and $E_i = \pi^{-1}(Z_i)$. Let
  $\pi_i : E_i \to Z_i$ be the restriction of $\pi$.

  Then $\pi^* : A^*(X) \to A^*(\wX)$ is injective, with image
  \[
    \{ \alpha \in A^*(\wX) \mid \alpha\vert_{E_i} \in \pi_i^*(A^*(Z_i)) \text{
      for all } i\}.
  \]
\end{lemma}

\subsection{Proof of Theorem \ref{ChowM}}

Write the small resolution $\wM_1(3,6)$ as $\pi : \wM_1(3,6) \to \oM(3,6)$.

\subsubsection{Preliminary description}

The small resolution $\pi : \wM_1(3,6) \to \oM(3,6)$ is an isomorphism away from
the 15 exceptional lines $L_{ij,kl,mn}$ mapping to the 15 points $P_{ij,kl,mn}$.
Any $P_{ij,kl,mn}$ is the birational image of any point in $L_{ij,kl,mn}$. Thus
$\pi$ satisfies the hypotheses of Lemma \ref{ChowEnvelope}. Since $P_{ij,kl,mn}$
is a point, $A^*(P_{ij,kl,mn}) = \bZ$, so by Lemma \ref{ChowEnvelope}, the image
of $A^*(\oM(3,6))$ in $A^*(\wM_1(3,6))$ is
\[
  \{\alpha \in A^*(\wM_1(3,6)) \mid \alpha\vert_{L_{ij,kl,mn}} \in \bZ \text{
    for all } L_{ij,kl,mn}\}.
\]
The first part of Theorem \ref{ChowM} now follows from Proposition \ref{RestrictionMaps}.

\subsubsection{Picard group}

From the formula
\[
  A^1(\oM(3,6)) = \{\alpha \in A^1(\wM_1(3,6)) \mid \alpha\vert_{L_{ij,kl,mn}} =
  0 \text{ for all } P_{ij,kl,mn}\}
\]
just established, together with the description of $A^1(\wM_1(3,6))$ from
Theorem \ref{ChowResolutions}, we immediately obtain the desired description of
$A^1(\oM(3,6))$ from part 3 of Theorem \ref{ChowM}. To finish proving part 3,
all that remains is to show
that $\Pic \oM(3,6) = A^1(\oM(3,6))$.

We have
\[
  \Pic \oM(3,6) \cong \pi^*\Pic \oM(3,6) \subset \Pic \wM_1(3,6) \cong
  A^1(\wM_1(3,6)),
\]
and furthermore if $\pi^*\alpha \in \pi^*\Pic \oM(3,6)$, then by the projection
formula
\[
  \pi^*\alpha \cdot L_{ij,kl,mn} = \alpha \cdot \pi_*L_{ij,kl,mn} = \alpha \cdot
  P_{ij,kl,mn} = 0,
\]
so $\pi^*\Pic \oM(3,6) \subset \pi^*A^1(\oM(3,6))$. To show equality, it
suffices to show that each generator of $A^1(\oM(3,6))$ is Cartier. This is
immediate from our description of $A^1(\oM(3,6))$ above, together with
Proposition \ref{BoundaryDivs}. This proves part 3 of Theorem \ref{ChowM}.

\subsubsection{Ranks of Chow groups}

Part 2 of Theorem \ref{ChowM} is immediate from parts 1 and 3 just established,
together with Theorem \ref{ChowResolutions}.

\subsubsection{Generators of $A^*(\oM(3,6))$}

It is a direct calculation that $A^*(\oM(3,6))$ is generated by $A^1(\oM(3,6))$.
(We performed this calculation on a computer.) This establishes part 4 of
Theorem \ref{ChowM}, and thus completes the proof of the Theorem.

\subsubsection{Relations on $A^*(\oM(3,6))$}

\begin{remark}
  Given the description of $A^*(\oM(3,6))$ as a subring of $A^*(\wM_1(3,6))$,
  one can determine the relations on $A^*(\oM(3,6))$ (thus a presentation for
  $A^*(\oM(3,6))$) by pulling back the relations on $A^*(\wM_1(3,6))$. However,
  the relations on $A^*(\oM(3,6))$ obtained in this manner are not as simple as
  the relations on $A^*(\wM_1(3,6))$, and it is easier to just work with
  $A^*(\oM(3,6))$ as a subring of $A^*(\wM_1(3,6))$.
\end{remark}

% \subsubsection{Presentation of Chow ring}
%
% One can directly check (e.g. with a computer) that the morphism
% \[
%   \bZ[\delta_{ijk,lmn},\delta_{ij,k,lmn},\delta_{ij,kl,mn}] \to
%   A^*(\wM_1(3,6)),
% \]
% defined as in part (2) of the theorem, has image $A^*(\oM(3,6)) \subset
% A^*(\wM_1(3,6))$ and kernel consisting of the desired relations of the
% theorem. (Notice that each relation in the theorem follows from the obvious
% relations on $A^*(\wM_1(3,6))$ coming from $\oM(3,6)$.)

\section{Tautological classes} \label{TautologicalClasses}

\subsection{Definitions}

For any $\oM(r,n)$, let $\pi : (\bS,\bB = \sum_{i=1}^n \bB_i) \to \oM(r,n)$
denote the universal family. By \cite[Proposition
5.1]{hackingCompactificationModuliSpace2006}, there are $\binom{n}{r-1}$
sections $\sigma_I: \oM(r,n) \to \bS$ of $\pi$, for $I \subset [n]$ with $\lvert
I \rvert = r-1$, with images $\bB_I = \bigcap_{i \in I} \bB_i$ in $\bS$.
Furthermore, at a fiber $(S,B = \sum B_i)$ of $\pi$, $S$ is smooth and $B$ has
normal crossings at the point $B_I = \bigcap_{i \in I} B_i$.

\begin{definition}
  Define $\bL_I = \sigma_I^*(\omega_{\pi})$ and $\phi_I = c_1(\bL_I)$.
\end{definition}

Observe that $\bL_I$ is a vector bundle whose fiber at a stable hyperplane
arrangement $(S,B)$ is the cotangent space to $S$ at $B_I$.

\begin{definition}
  For $i \in I$, define $\bL_{I,i} = \sigma_I^*(\omega_{\pi}\vert_{\bB_i})$ and
  $\psi_{I,i} = c_1(\bL_{I,i})$.
\end{definition}

By adjunction, the curve $C_{I \setminus i} = \bigcap_{j \in I \setminus i} B_j$
on a stable hyperplane arrangment $(S,B)$ is a stable $(n-r+2)$-pointed curve of
genus zero, where the marked points are $P_k = B_k \cap C_{I \setminus i}$ for
$k \not\in I \setminus i$. Observe that $\bL_{I,i}$ is a line bundle whose fiber
at $(S,B)$ is the cotangent line to $C_{I \setminus i}$ at $B_I$. There is a
decomposition
\[
  \bL_I = \bigoplus_{i \in I} \bL_{I,i}, \;\; \phi_I = \sum_{i \in I}
  \psi_{I,i}.
\]
To understand the vector bundle $\bL_I$ and its first Chern class $\phi_I$, it
is therefore enough to understand the individual line bundles $\bL_{I,i}$ and
their first Chern classes $\psi_{I,i}$.

\begin{example}
  When $r = 2$, $I = \{i\}$, and we write $\psi_{I,i} = \psi_i$. Then $\phi_I =
  \phi_i = \psi_i$ is just the usual $\psi$-class on $\oM_{0,n}$. The linear
  system $\lvert \psi_i \rvert$ defines Kapranov's birational morphism $q_i :
  \oM_{0,n} \to \bP^{n-3}$ \cite{kapranovChowQuotientsGrassmannians1993}.
\end{example}

\begin{example} \label{psi3} When $r=3$, $I = \{i,j\}$, and we write $\psi_{I,j}
  = \psi_{ij}$. Note that $\psi_{ij} = r_i^*(\psi_j)$. The class $\phi_{ij} =
  \psi_{ij} + \psi_{ji}$ is a symmetric version of $\psi_{ij}$. The linear
  system $\lvert \phi_{ij} \rvert$ defines a birational morphism $q_{ij} :
  \oM(3,n) \to \bP^{n-4} \times \bP^{n-4}$. We considered various $q_{ij} :
  \oM(3,6) \to \bP^2 \times \bP^2$ previously.
\end{example}

\begin{example} \label{psigeneral} Generalizing the previous examples, on any
  $\oM(r,n)$ one can write $\psi_{I,i} = r_{I \setminus i}^*(\psi_{i})$, where
  $r_{I \setminus i}: \oM(r,n) \to \oM_{0,n-r+2}$ is the restriction to the
  curve $C_{I \setminus i}$. The linear system $\lvert \phi_I \rvert$ defines a
  birational morphism $q_I : \oM(r,n) \to (\bP^{n-r-1})^{r-1}$.
\end{example}

\subsection{Intersections of $\psi$-classes}

Motivated by the case of curves, where the top intersections of any tautological
classes on a given $\oM_{g,n}$ are governed by the top intersections of the
$\psi$-classes on all $\oM_{g,n}$
\cite{faberAlgorithmsComputingIntersection1999}, and in turn the top
intersections of the $\psi$-classes are governed by Witten's conjecture
\cite{wittenTwoDimensionalGravityIntersection1991},
\cite{kontsevichIntersectionTheoryModuli1992}, we seek a method of determining
top intersections of $\psi$-classes on $\oM(r,n)$.

\subsubsection{$\psi$-classes on $\oM_{0,n}$}

On $\oM_{0,n}$ (and more generally on $\oM_{g,n}$), the $\psi$-classes are
defined by
\[
  \psi_i = c_1(\sigma_i^*(\omega_{\pi})),
\]
where $\pi : \oM_{g,n+1} \to \oM_{g,n}$ is the universal family, with $n$
sections $\sigma_i$. These are determined recursively by the pullback formula:
\[
  \psi_i = f_k^*(\psi_i) + D_{ik},
\]
where $i,k$ are distinct, $f_k : \oM_{0,n+1} \to \oM_{0,n}$ is the $k$th
forgetful map, and $D_{ik}$ is the divisor parameterizing the stable curve with
$i,k$ on one irreducible component and the remaining marked points on the other
\cite{wittenTwoDimensionalGravityIntersection1991}.

From the pullback formula one obtains the expression
\[
  \psi_i = \sum_{i \in I,j,k \in J} D_{I,J}
\]
for $\psi_i$ as a sum of boundary divisors on $\oM_{0,n}$ \cite[Section
4]{getzlerTopologicalRecursionRelations1998a}.

The pullback formula also implies the string equation:
\[
  \int \psi_1^{k_1}\cdots\psi_n^{k_n} \cap [\oM_{0,n+1}] = \sum_{i=1}^n \int
  \psi_1^{k_1}\cdots \psi_i^{k_i-1} \cdots \psi_n^{k_n} \cap [\oM_{0,n}].
\]
Together with the initial condition $\int_{\oM_{0,3}} \psi_i = 1$, this allows
one to compute the following formula
\cite{wittenTwoDimensionalGravityIntersection1991}.
\[
  \int \psi_1^{k_1}\cdots \psi_n^{k_n} \cap [\oM_{0,n}] =
  \binom{n-3}{k_1,\ldots,k_n}
\]

\subsubsection{$\psi$-classes on $\oM(r,n)$}

\begin{lemma}[Pullback formula] \label{PullbackGeneral} On $\oM(r,n)$ ($r \geq
  3$), we have
  \[
    \psi_{I,i} = f_k^*(\psi_{I,i}) + r_{I \setminus i}^*(D_{ik})
  \]
\end{lemma}

\begin{proof}
  The result is trivial for $n=r+1$, because $\oM(r,r+1)$ is a point.

  For $n \geq r+2$, the following diagram commutes.
  \[
    \begin{tikzcd}
      \oM(r,n) \ar[d, "f_k"] \ar[r, "r_{I \setminus i}"] & \oM_{0,n-r+2} \ar[d, "f_k"] \\
      \oM(r,n-1) \ar[r, "r_{I \setminus i}"] & \oM_{0,n-r+1}
    \end{tikzcd}
  \]

  The pullback formula $\psi_i = f_k^*(\psi_i) + D_{ik}$ for $\oM_{0,n}$
  together with the formula $\psi_{I,i} = r_{I \setminus i}^*(\psi_i)$ (Example
  \ref{psigeneral}) gives
  \[
    \psi_{I,i} = r_{I \setminus i}^*(f_k^*(\psi_i) + D_{ik}) = r_{I \setminus
      i}^*f_k^*\psi_i + r_{I \setminus i}^*D_{ik} \text{ on } \oM(r,n).
  \]
  Commutativity implies that
  \[
    r_{I \setminus i}^*f_k^*\psi_i = f_k^*(r_{I \setminus i}^*\psi_i) =
    f_k^*(\psi_{I, i}),
  \]
  so the result follows.
\end{proof}

Unfortunately, the intersections of the $r_{I \setminus i}^*(D_{ik})$ are
typically nonzero, which makes recursive computations of intersections of
$\psi$-classes on $\oM(r,n)$ more complicated than in the rank 2 case.

\subsubsection{$\psi$-classes on $\oM(3,n)$}

Using our notation $\psi_{ij}$ for the $\psi$-classes on $\oM(3,n)$ (Example
\ref{psi3}), the pullback formula \ref{PullbackGeneral} for $\oM(3,n)$ takes the
form
\[
  \psi_{ij} = f_k^*(\psi_{ij}) + r_i^*(D_{jk}) \text{ for } i,j,k \text{
    distinct.}
\]

\begin{example} \label{psiM35} On $\oM(3,5)$, we have
  \[
    \psi_{ij} = r_i^*(D_{jk}) = D_{jk,ilm} + D_{lm,ijk},
  \]
  where $r_i : \oM(3,5) \to \oM_{0,4}$ is a restriction map. Since any two
  points on $\oM_{0,4} \cong \bP^1$ are linearly equivalent, it follows that
  $\psi_{ij}$ and $\psi_{ik}$ are linearly equivalent for any $j,k$.

  Under the duality $\oM(3,5) \cong \oM_{0,5}$, we can write
  \[
    \psi_{ij} = \frac{1}{3}\sum_{k=1}^5 \psi_k - \psi_i,
  \]
  where $\psi_k$ are the usual $\psi$-classes on $\oM_{0,5}$. The intersection
  numbers of the $\psi_{ij}$ are given by
  \begin{align*}
    \int \psi_{i_1j_1}\psi_{i_2j_2} \cap [\oM(3,5)] =
    \begin{cases}
      0, & i_1 = i_2,\\
      1, & i_1 \neq i_2.
    \end{cases}
  \end{align*}
\end{example}

\begin{theorem} \label{M36intersectionnumbers} Let $M = \oM(3,6)$ or any of its
  small resolutions. Any intersection number on $M$ can be determined by the
  formula
  \[
    \int \psi_{56}^2\psi_{65}^2 \cap [M] = 1.
  \]
\end{theorem}

\begin{proof}
  Observe that $\psi_{56}$ and $\psi_{65}$ are the pullbacks of the generators
  of $A^*(\bP^2 \times \bP^2)$ via the map $q_{56} : M \to \bP^2 \times \bP^2$.
  The formula $\int \psi_{56}^2\psi_{65}^2 \cap [M] = 1$ follows from the
  corresponding formula on $A^*(\bP^2 \times \bP^2)$. Since $\rk A^4(M) = 1$,
  any top intersection on $M$ is necessarily a multiple of
  $\psi_{56}^2\psi_{65}^2$.
\end{proof}

\begin{theorem} \label{psiIntersectionsM36} On $\oM(3,6)$, one has
  \[
    \psi_{ij} = f_n^*(D_{jk}) + f_n^*(D_{lm}) + r_i^*(D_{jn}),
  \]
  and
  \[
    \psi_{i_1j_1}\cdots \psi_{i_4j_4} = 0 \iff \geq 3 \text{ of the }
    i_k\text{'s coincide} .
  \]
  The nonzero intersection numbers of the $\psi_{ij}$ are listed, up to
  $S_6$-symmetry, in the below table.
  \begin{center}
    \begin{tabular}{l | l}
      $n$ & Product \\
      \hline
      1 & $\psi_{12}\psi_{12}\psi_{21}\psi_{21}$, $\psi_{12}\psi_{12}\psi_{21}\psi_{31}$, $\psi_{12}\psi_{12}\psi_{23}\psi_{23}$, $\psi_{12}\psi_{12}\psi_{23}\psi_{32}$, $\psi_{12}\psi_{12}\psi_{31}\psi_{41}$, $\psi_{12}\psi_{12}\psi_{32}\psi_{32}$,  $\psi_{12}\psi_{12}\psi_{34}\psi_{34}$,\\
          & $\psi_{12}\psi_{12}\psi_{34}\psi_{43}$\\
      \hline
      2 & $\psi_{12}\psi_{12}\psi_{21}\psi_{23}$, $\psi_{12}\psi_{12}\psi_{21}\psi_{32}$, $\psi_{12}\psi_{12}\psi_{21}\psi_{34}$, $\psi_{12}\psi_{12}\psi_{23}\psi_{24}$, 
          $\psi_{12}\psi_{12}\psi_{23}\psi_{31}$, $\psi_{12}\psi_{12}\psi_{23}\psi_{34}$,
          $\psi_{12}\psi_{12}\psi_{23}\psi_{41}$,\\
          &$\psi_{12}\psi_{12}\psi_{23}\psi_{42}$, $\psi_{12}\psi_{12}\psi_{23}\psi_{43}$,  $\psi_{12}\psi_{12}\psi_{31}\psi_{32}$, $\psi_{12}\psi_{12}\psi_{31}\psi_{34}$, $\psi_{12}\psi_{12}\psi_{31}\psi_{42}$,
            $\psi_{12}\psi_{12}\psi_{31}\psi_{43}$, $\psi_{12}\psi_{12}\psi_{31}\psi_{45}$,\\
          &$\psi_{12}\psi_{12}\psi_{32}\psi_{34}$,  $\psi_{12}\psi_{12}\psi_{32}\psi_{42}$, 
            $\psi_{12}\psi_{12}\psi_{32}\psi_{43}$, $\psi_{12}\psi_{12}\psi_{34}\psi_{35}$,
            $\psi_{12}\psi_{12}\psi_{34}\psi_{45}$, $\psi_{12}\psi_{12}\psi_{34}\psi_{54}$, 
            $\psi_{12}\psi_{13}\psi_{21}\psi_{31}$,  \\
          &$\psi_{12}\psi_{13}\psi_{21}\psi_{41}$, $\psi_{12}\psi_{13}\psi_{23}\psi_{32}$, $\psi_{12}\psi_{13}\psi_{24}\psi_{42}$,
            $\psi_{12}\psi_{13}\psi_{41}\psi_{51}$, $\psi_{12}\psi_{13}\psi_{45}\psi_{54}$\\
      \hline
      3 & $\psi_{12}\psi_{12}\psi_{23}\psi_{45}$, $\psi_{12}\psi_{12}\psi_{32}\psi_{45}$, $\psi_{12}\psi_{12}\psi_{34}\psi_{56}$, $\psi_{12}\psi_{21}\psi_{31}\psi_{41}$, $\psi_{12}\psi_{21}\psi_{34}\psi_{43}$, $\psi_{12}\psi_{23}\psi_{42}\psi_{52}$, $\psi_{12}\psi_{32}\psi_{42}\psi_{52}$\\
      \hline
      4 & $\psi_{12}\psi_{13}\psi_{21}\psi_{23}$, $\psi_{12}\psi_{13}\psi_{21}\psi_{24}$, $\psi_{12}\psi_{13}\psi_{21}\psi_{32}$,
          $\psi_{12}\psi_{13}\psi_{21}\psi_{34}$, $\psi_{12}\psi_{13}\psi_{21}\psi_{42}$, $\psi_{12}\psi_{13}\psi_{21}\psi_{43}$,
          $\psi_{12}\psi_{13}\psi_{21}\psi_{45}$, \\
          & $\psi_{12}\psi_{13}\psi_{23}\psi_{24}$, $\psi_{12}\psi_{13}\psi_{23}\psi_{34}$,
            $\psi_{12}\psi_{13}\psi_{23}\psi_{41}$, $\psi_{12}\psi_{13}\psi_{23}\psi_{42}$, $\psi_{12}\psi_{13}\psi_{23}\psi_{43}$,
            $\psi_{12}\psi_{13}\psi_{24}\psi_{25}$, $\psi_{12}\psi_{13}\psi_{24}\psi_{34}$, \\
          &$\psi_{12}\psi_{13}\psi_{24}\psi_{41}$,
            $\psi_{12}\psi_{13}\psi_{24}\psi_{43}$, $\psi_{12}\psi_{13}\psi_{24}\psi_{45}$, $\psi_{12}\psi_{13}\psi_{24}\psi_{51}$,
            $\psi_{12}\psi_{13}\psi_{24}\psi_{52}$, $\psi_{12}\psi_{13}\psi_{24}\psi_{54}$, $\psi_{12}\psi_{13}\psi_{41}\psi_{52}$,\\
          &$\psi_{12}\psi_{13}\psi_{41}\psi_{54}$, $\psi_{12}\psi_{13}\psi_{41}\psi_{56}$, $\psi_{12}\psi_{13}\psi_{42}\psi_{43}$,
            $\psi_{12}\psi_{13}\psi_{42}\psi_{45}$, $\psi_{12}\psi_{13}\psi_{42}\psi_{52}$, $\psi_{12}\psi_{13}\psi_{42}\psi_{54}$,
            $\psi_{12}\psi_{13}\psi_{45}\psi_{46}$, \\
          &$\psi_{12}\psi_{13}\psi_{45}\psi_{56}$, $\psi_{12}\psi_{13}\psi_{45}\psi_{65}$,
            $\psi_{12}\psi_{21}\psi_{31}\psi_{42}$, $\psi_{12}\psi_{21}\psi_{31}\psi_{43}$\\
      \hline
      5 & $\psi_{12}\psi_{21}\psi_{34}\psi_{45}$, $\psi_{12}\psi_{21}\psi_{31}\psi_{45}$, $\psi_{12}\psi_{23}\psi_{34}\psi_{53}$, $\psi_{12}\psi_{23}\psi_{42}\psi_{53}$, $\psi_{12}\psi_{23}\psi_{43}\psi_{53}$\\
      \hline
      6 & $\psi_{12}\psi_{13}\psi_{23}\psi_{45}$, $\psi_{12}\psi_{13}\psi_{24}\psi_{35}$, $\psi_{12}\psi_{13}\psi_{24}\psi_{53}$,
          $\psi_{12}\psi_{13}\psi_{24}\psi_{56}$, $\psi_{12}\psi_{13}\psi_{42}\psi_{53}$, $\psi_{12}\psi_{13}\psi_{42}\psi_{56}$,
          $\psi_{12}\psi_{21}\psi_{34}\psi_{54}$, \\
          &$\psi_{12}\psi_{21}\psi_{34}\psi_{56}$, $\psi_{12}\psi_{23}\psi_{31}\psi_{41}$,
            $\psi_{12}\psi_{23}\psi_{34}\psi_{41}$, $\psi_{12}\psi_{23}\psi_{34}\psi_{52}$, $\psi_{12}\psi_{23}\psi_{42}\psi_{56}$,
            $\psi_{12}\psi_{23}\psi_{43}\psi_{54}$, $\psi_{12}\psi_{32}\psi_{42}\psi_{56}$\\
      \hline
      7 & $\psi_{12}\psi_{23}\psi_{34}\psi_{45}$, $\psi_{12}\psi_{23}\psi_{34}\psi_{54}$, $\psi_{12}\psi_{23}\psi_{43}\psi_{56}$, $\psi_{12}\psi_{23}\psi_{45}\psi_{56}$\\
      \hline
      8 & $\psi_{12}\psi_{23}\psi_{34}\psi_{56}$, $\psi_{12}\psi_{23}\psi_{45}\psi_{65}$, $\psi_{12}\psi_{32}\psi_{45}\psi_{65}$\\
      \hline
      9 & $\psi_{12}\psi_{23}\psi_{31}\psi_{45}$
    \end{tabular}
    \captionof{table}{Intersections of $\psi$-classes on $\oM(3,6)$}
  \end{center}
\end{theorem}

\begin{proof}
  The expression for $\psi_{ij}$ follows from the pullback formula
  \[
    \psi_{ij} = f_n^*(\psi_{ij}) + r_i^*(D_{jn})
  \]
  together with the expression for $\psi_{ij}$ on $\oM(3,5)$ from Example
  \ref{psiM35}. Note that by Proposition \ref{PullbackFormulas}, this gives an
  explicit expression for $\psi_{ij}$ as a sum of boundary divisors on
  $\oM(3,6)$.

  The intersection products of the $\psi_{ij}$ can be computed on
  $A^*(\oM(3,6))$ using Theorem \ref{ChowM}. Alternatively, the $\psi_{ij}$ are
  all disjoint from the singular locus of $\oM(3,6)$, so the computations can
  also be performed on any $A^*(\wM_{S_1,S_2}(3,6))$ using Theorem
  \ref{ChowResolutions}. In turn the intersection numbers are determined by
  Theorem \ref{M36intersectionnumbers}. We performed these calculations on
  $A^*(\wM_1(3,6))$ using a computer.
\end{proof}

\begin{question}
  Is there a nice combinatorial formula for the intersection numbers of the
  $\psi$-classes on $\oM(3,n)$ (more generally, on $\oM(r,n)$)?
\end{question}

\section{Birational geometry} \label{BirationalGeometry}

This section is over $\bC$.

The intersection-theoretic computations in this section are performed with
coefficients in $\bQ$. Set $\wM_1 = \wM_1(3,6)$, $\wB_1 = \wM_1 \setminus
M(3,6)$, $\oM = \oM(3,6)$, $B = \oM \setminus M(3,6)$.

\begin{proposition} \label{lcClass}
  \begin{enumerate}
  \item $K_{\wM_1} = -\frac{3}{10} \sum D_{ijk,lmn} - \frac{1}{5}\sum
    D_{ij,klmn} + \frac{1}{5}\sum D_{ij,kl,mn}$
  \item $K_{\wM_1} + \wB_1 = \frac{7}{10} \sum D_{ijk,lmn} + \frac{4}{5}\sum
    D_{ij,klmn} + \frac{6}{5}\sum D_{ij,kl,mn}$
  \end{enumerate}
\end{proposition}

\begin{proof}
  The second part is immediate from the first.

  The first part is obtained from the blowup construction $\wM_1 \to \oM_{0,5}
  \times \oM_{0,5}$ of Theorem \ref{BlowupConstruction}. The canonical class of
  $\oM_{0,5}$ is $K_{\oM_{0,5}} = -\frac{1}{2}B_{\oM_{0,5}}$ \cite[Lemma
  3.5]{keelContractibleExtremalRays1996}, so
  \[
    K_{\oM_{0,5} \times \oM_{0,5}} = -\frac{1}{2}(p_1^*(B_{\oM_{0,5}}) +
    p_2^*(B_{\oM_{0,5}})).
  \]
  From the blowup sequence $\wM_1 \to X_1$, we find that
  \[
    K_{\wM_1} = K_{\oM_{0,5}} \times \oM_{0,5} + D =
    -\frac{1}{2}(r_6^*(B_{\oM_{0,5}}) + r_5^*(B_{\oM_{0,5}})) + D
  \]
  where
  \begin{align*}
    D &= \sum D_{ijk,l56} + \sum (D_{ij,kl56} + D_{ij,kl,56}) + D_{56,1234} + \sum D_{ij,kl,mn}.
  \end{align*}
  The result follows by applying the linear relations on $A^*(\wM_1(3,6))$.
\end{proof}

\begin{theorem} \label{logcanonicalM} The pair $(\oM(3,6),B)$ is log canonical,
  and the log canonical divisor $K_{\oM(3,6)} + B$ is ample and Cartier.
\end{theorem}

This theorem was previously proven by Luxton
\cite{luxtonLogCanonicalCompactification2008} using tropical methods. We give an
independent proof relying only on explicit birational geometry.

\begin{lemma} \label{keylemma}\
  \begin{enumerate}
  \item Let $(Z,B_Z)$ be a smooth projective variety $Z$ with boundary $B_Z$ a
    normal crossings divisor.
  \item Assume there is an effective nef divisor $D$ on $Z$ with $\Supp D =
    B_Z$.
  \item Suppose $K_Z + B_Z$ is nef and big.
  \end{enumerate}
  If all of the above hold, then for $m \gg 0$, $\lvert m(K_Z + B_Z) \rvert$ is
  basepoint-free, and gives a contraction $\varphi : Z \to Z_{lc}$ to the log
  canonical model of $(Z,B_Z)$. Furthermore, the locus contracted by $\varphi$
  is exactly the locus $L$ where $(K_Z + B_Z)\vert_L = 0$.
\end{lemma}

\begin{proof}
  Let $\Delta = B_Z - \epsilon D$. By construction this looks like $\sum
  b_iB_i$, with $0 < b_i < 1$, where $B_i$ are the irreducible components of
  $B_Z$. (We choose $\epsilon$ small enough to ensure $0 < b_i < 1$.) Since
  $B_Z$ is a normal crossings divisor, the pair $(Z,\Delta)$ is klt.

  Now
  \begin{align*}
    m(K_Z + B_Z) - (K_Z + \Delta) = (m-1)(K_Z + B_Z) + \epsilon \Delta
  \end{align*}
  is the sum of a big and nef divisor, and a nef divisor, hence is big and nef.
  Then $\lvert m(K_Z + B_Z) \rvert$ is basepoint-free by the Basepoint-Free
  Theorem \cite[Theorem 3.3]{kollarBirationalGeometryAlgebraic1998}, and by
  definition it gives a contraction $f : (Z,B) \to (Z_{lc},B_{lc})$ to the log
  canonical model of $(Z,B_Z)$.

  By definition, the log canonical class $K_{Z_{lc}} + B_{Z_{lc}}$ of the log
  canonical model is ample. We have $K_Z + B_Z = f^*(K_{Z{lc}} + B_{Z_{lc}})$,
  so the locus $L$ where $K_Z + B_Z$ is not ample must be exactly the locus
  contracted by $f$. Since $K_Z + B_Z$ is nef, $L$ is exactly the locus where
  $(K_Z + B_Z)\vert_{L} = 0$.
\end{proof}

\begin{lemma} \label{nefcriterion} Let $(Z,B_Z)$ be a smooth projective variety
  $Z$ with boundary $B_Z$ a normal crossings divisor. Suppose $K_Z + B_Z$ is
  linearly equivalent to an effective divisor with support $B_Z$. Then $K_Z +
  B_Z$ is nef if and only if $(K_Z + B_Z)\vert_D = K_D + B_D$ is nef for all
  irreducible components $D$ of $B_Z$.
\end{lemma}

\begin{proof}
  ($\implies$) Immediate.
  
  ($\impliedby$) Suppose $K_D + B_D$ is nef for all irreducible components $D$
  of $B_Z$.

  Let $C$ be any irreducible curve in $Z$.
  \begin{enumerate}
  \item If $C \subset B_Z$, then $C \subset D$ for some irreducible boundary
    divisor $D$. Since $K_D + B_D$ is nef, it follows that
    \[
      (K_Z + B_Z) \cdot L = (K_Z + B_Z)\vert_D \cdot L = (K_D + B_D)L \geq 0.
    \]
  \item If $C \not\subset B_Z$, then $C \not\subset D$ for any boundary divisor.
    Because $K_Z + B_Z$ is an effective sum of boundary divisors, it follows
    that $(K_Z + B_Z)C \geq 0$.
  \end{enumerate}
\end{proof}

\subsection{Proof of Theorem \ref{logcanonicalM}}
We will prove Theorem \ref{logcanonicalM} by applying Lemma \ref{keylemma} to
$\wM_1 = \wM_1(3,6)$.

\subsubsection{$\wM_1$ is a smooth projective variety and $\wB_1$ is a normal
  crossings divisor} This is immediate from Proposition \ref{SmallResolutions}.

\subsubsection{There is an effective nef divisor $D$ on $\wM_1$ with $\Supp D = \wB_1$}
Recall the morphism $f:\wM_1 \to X_0 = \bP^2 \times \bP^2$ constructed in
Theorem \ref{BlowupConstruction}. Let $D_0 \subset \bP^2 \times \bP^2$ be the
sum of the boundary divisors $L_{ij} \times \bP^2$, $\bP^2 \times L_{ij}$,
$Q_{ijk}$ in $\bP^2 \times \bP^2$.

\begin{lemma} \label{D1props}
  \begin{enumerate}
  \item $D_0$ contains all the blown up centers of $f:\wM_1 \to X_0$.
  \item $D_0$ is ample.
  \end{enumerate}
\end{lemma}

\begin{proof}
  The first part is immediate.

  Let $H_1 = c_1(\cO(1,0)),H_2 = c_1(\cO(0,1))$ be the two hyperplane classes on
  $\bP^2 \times \bP^2$. Then $D_0 \sim 10H_1 + 10H_2$, hence is ample.
\end{proof}

Let $D = f^*D_0$. It follows from the lemma that $D$ is an effective nef divisor
on $\wM_1$ with $\Supp D = \wB_1$.

\subsubsection{$K_{\wM_1} + \wB_1$ is nef and big} \label{NefBigSmallRes}

\begin{lemma} \label{nefcanonical} $K_{\wM_1} + \wB_1$ is nef, and zero exactly
  on the exceptional lines $L_{ij,kl,mn}$.
\end{lemma}

\begin{proof}
  By Proposition \ref{lcClass}, $K_{\wM_1} + \wB_1$ is linearly equivalent to an
  effective sum of boundary divisors.

  As in the proof of Lemma \ref{nefcriterion}, if $C \not\subset \wB_1$, then
  $(K_{\wM_1} + \wB_1)C \geq 0$. Furthermore, $(K_{\wM_1} + \wB_1)C = 0 \iff C
  \cap \wB_1 = \emptyset$. But if $C \cap \wB_1 = \emptyset$, then $f(C) \cap
  f(\wB_1) = \emptyset$, where $f : \wM_1 \to X_0$ as above. Since $C
  \not\subset \wB_1$, $f(C)$ is still a curve, but $f(\wB_1) = D_0$ is ample, so
  we must have $D_0f_*(C) > 0$. We conclude that any curve in $\wM_1$ meets the
  boundary, and $(K_{\wM_1} + \wB_1)C > 0$ for $C \not\subset \wB_1$.

  By Lemma \ref{nefcriterion}, to show $K_{\wM_1} + \wB_1$ is nef it is enough
  to show that $(K_{\wM_1} + \wB_1)\vert_D = K_D + B_D$ is nef for all
  irreducible boundary divisors $D$ of $\wM_1$.
  \begin{enumerate}
  \item If $D = D_{ijk,lmn}$, then $D \cong \oM_{0,6}$ with its natural
    boundary, so $K_{D} + B_D$ is ample by \cite[Lemma
    3.6]{keelContractibleExtremalRays1996}.
  \item If $D = D_{ij,kl,mn}$, then $D \cong (\bP^1)^3$ with boundary $p \times
    \bP^1 \times \bP^1$, $\bP^1 \times p \times \bP^1$, $\bP^1 \times \bP^1
    \times p$, for $p=0,1,\infty$. Thus $K_D + B_D = h_1 + h_2 + h_3$, where
    $h_i$ is the pullback along the $i$th projection of a hyperplane class in
    $\bP^1$. In particular, $K_D + B_D$ is ample.
  \item If $D = D_{ij,klmn}$, then $D$ is a small resolution of $\oM_{0,5}
    \times_{\oM_{0,4}} \oM_{0,5}$. By Knudsen's construction, $\oM_{0,6}$ is a
    blowup of $\oM_{0,5} \times_{\oM_{0,4}} \oM_{0,5}$ along a locus contained
    in the boundary. Furthermore, the boundary divisors of $\oM_{0,6}$ are the
    strict transforms of the boundary divisors of $D$, together with the
    exceptional divisors. Since $K_{\oM_{0,6}} + B_{\oM_{0,6}}$ is linearly
    equivalent to an effective sum of boundary divisors on $\oM_{0,6}$, it
    follows that $K_{D} + B_D$ is linearly equivalent to an effective sum of
    boundary divisors on $D$. Thus by Lemma \ref{nefcriterion}, it is enough to
    show that $(K_D + B_D)\vert_{D'} = K_{D'} + B_{D'}$ for any boundary divisor
    in $D$. There are three cases.
    \begin{enumerate}
    \item If $D' = D \cap D_{abc,def}$, then $D'$ is also a boundary divisor in
      $D_{abc,def}$, hence $K_{D'} + B_{D'}$ is ample.
    \item If $D' = D \cap D_{ab,cd,ef}$, then $D'$ is also a boundary divisor in
      $D_{ab,cd,ef}$, so again $K_{D'} + B_{D'}$ is ample.
    \item If $D' = D \cap D_{ab,cdef}$, then $D' \cong Bl_1(\bP^1 \times
      \bP^1)$. The boundary of $D'$ consists of the strict transforms of $p
      \times \bP^1, \bP^1 \times p$, for $p=0,1,\infty$, as well as the
      exceptional line on $Bl_1(\bP^1 \times \bP^1)$. We have
      \[
        K_{D'} + B_{D'} = h_1 + h_2,
      \]
      which is nef, and zero exactly on the exceptional line. Note the
      exceptional line is one of the exceptional lines $L_{ij,kl,mn}$ on
      $\wM_1(3,6)$.
    \end{enumerate}
    Finally, by a similar argument to the beginning of this proof, $(K_D + B_D)C
    > 0$ for any irreducible curve $C \subset D$ which is not contained in the
    boundary of $D$.
  \end{enumerate}
  From the above we conclude that $K_{\wM_1} + \wB_1$ is nef, and vanishes
  exactly on the exceptional lines $L_{ij,kl,mn}$, as desired.
\end{proof}

\begin{lemma}
  $K_{\wM_1} + \wB_1$ is big.
\end{lemma}

\begin{proof}
  Since $K_{\wM_1} + \wB_1$ is nef, it suffices to show that $(K_{\wM_1} +
  \wB_1)^4 > 0$. This is an immediate calculation in the Chow ring of $\wM_1$,
  using Theorems \ref{ChowResolutions} and \ref{M36intersectionnumbers}.
  % \cite[Theorem 2.2.16]{Lazarsfeld}.
\end{proof}

\subsubsection{Completion of proof}
By the previous subsections, we can apply Lemma \ref{keylemma} to
$(\wM_1,\wB_1)$. We obtain a contraction $\wM_1 \to \oM_{lc}$ to the log
canonical model, given by the basepoint-free linear system $\lvert m(K_{\wM_1} +
\wB_1) \rvert$ for $m \gg 0$. This contraction is an isomorphism everywhere
except the 15 exceptional lines $L_{ij,kl,mn}$ on $\wM_1$, which it contracts to
points. Since the map $\wM_1 \to \oM$ has the same description (Proposition
\ref{SmallResolutions}), we conclude that the normalization of $\oM$ is the log
canonical model $\oM_{lc}$. Since $\oM$ is already normal (by our standing
assumption), $\oM = \oM_{lc}$. Since $K_{\wM_1} + \wB_1 = \pi^*(K_{\oM} + B)$ is
nef and zero only on the exceptional lines, $K_{\oM} + B$ is ample.

\bibliography{IntersectionTheoryofM36} \bibliographystyle{amsalpha}
\end{document}